\documentclass[a4paper, 10pt]{article}

\usepackage{packageList}
\usepackage{macros}

\newcommand{\dist}{\mathrm{dist}}

\graphicspath{{./Figures/}}
\pgfplotsset{compat=1.11}
\newlength\fwidth

\title{Analysis of target data-dependent greedy kernel algorithms: Convergence rates for $f$-, $f \cdot P$- and $f/P$-greedy}
\author[1]{Tizian Wenzel \thanks{tizian.wenzel@mathematik.uni-stuttgart.de, corresponding author}}
\author[2]{Gabriele Santin \thanks{gsantin@fbk.eu, \href{http://orcid.org/0000-0001-6959-1070}{orcid.org/0000-0001-6959-1070}}}
\author[1]{Bernard Haasdonk \thanks{haasdonk@mathematik.uni-stuttgart.de}}
\affil[1]{Institute for Applied Analysis and Numerical Simulation, University of Stuttgart, Germany}
\affil[2]{Digital Society Center, Bruno Kessler Foundation, Trento, Italy}

\begin{document}

\maketitle
  
\begin{abstract}
Data-dependent greedy algorithms in kernel spaces are known to provide fast converging interpolants, while being extremely easy to 
implement and efficient to run.
Despite this experimental evidence, no detailed theory has yet been presented. This situation is unsatisfactory especially when compared to the case of 
the data-independent $P$-greedy algorithm, for which optimal convergence rates are available, despite its performances being usually inferior to the ones 
of target data-dependent algorithms.

In this work we fill this gap by first defining a new scale of greedy algorithms for interpolation that comprises all the existing ones in a unique analysis, 
where the degree of dependency of the selection criterion on the functional data is quantified by a real parameter. We then prove new convergence rates where 
this degree is taken into account and we show that, possibly up to a logarithmic factor, target data-dependent selection strategies provide faster convergence. 

In particular, for the first time we obtain convergence rates for target data adaptive interpolation that are faster than the ones given by uniform points, without the 
need of any special assumption on the target function. The rates are confirmed by a number of examples.

These results are made possible by a new analysis of greedy algorithms in general Hilbert spaces.

\end{abstract}

\section{Introduction} \label{sec:introduction}
Kernel methods are a well understood and widely used technique for approximation, regression and classification in machine learning and numerical analysis. 

We start by collecting some notation and preliminary results, while more details are provided in Section \ref{sec:background}. 
For a non-empty set $\Omega$ a kernel is defined as a symmetric function $k: \Omega \times \Omega \rightarrow \R$.
The kernel matrix $A_{X_n}$ for a set of points $X_n 
= \{ x_1, \dots, x_n \} \subset \Omega$ is 
given as $(A_{X_n})_{ij} = (k(x_i, x_j))_{ij} \in \R^{n \times n}$, $i,j=1, \dots, n$. 
If the kernel matrix is strictly positive definite for any set $X_n \subset \Omega$ of $n$ distinct points, the kernel is called {\emph{strictly positive 
definite}}.
Associated to every strictly positive definite kernel there is a unique {\emph{Reproducing Kernel Hilbert Space}} $\ns$ (RKHS) with inner product $\langle 
\cdot, \cdot \rangle_{\ns}$, which is also called {\emph{native space}} of $k$, and which is a space of real valued %
functions on 
$\Omega$ where the kernel $k$ acts as a {\emph{reproducing kernel}}, that is
\begin{enumerate}
\item $k(\cdot, x) \in \ns ~ \forall x \in \Omega$,
\item $f(x) = \langle f, k(\cdot, x) \rangle_{\ns} ~ \forall x \in \Omega,  \forall f \in \ns$ \hfill (reproducing property).
\end{enumerate}
Strictly positive definite continuous kernels can be used for the interpolation of continuous functions. The theory is developed under the assumption that $f \in \ns$, 
and in this case for any set of pairwise distinct interpolation points $X_n \subset \Omega$ there exists a unique minimum-norm interpolant $s_n \in \ns$ that 
satisfies
\begin{align}
\label{eq:interpol_conditions}
s_n(x_i) = f(x_i) ~~ \forall i=1, \dots, n. 
\end{align}
It can be shown that this interpolant is given by the orthogonal projection $\Pi_{V(X_n)}$ of $f$ onto the linear subspace $V(X_n) := \Sp \{ k(\cdot, x_i),\; 
x_i \in 
X_n \}$, i.e.,
\begin{align*}
s_n(\cdot) = \Pi_{V(X_n)}(f) = \sum_{j=1}^n \alpha_j k(\cdot, x_j).
\end{align*}
The coefficients $\alpha_j, j= 1, \dots,n$, can be calculated by solving the linear system of equations arising from the interpolation conditions in Eq. 
\eqref{eq:interpol_conditions}, which is always invertible due to the assumed strict positive definiteness of the kernel. \\
A standard way of estimating the error between the function $f$  and the interpolant in the $\Vert \cdot \Vert_{L^\infty(\Omega)}$-norm makes use of the 
{\emph{power 
function}}, which is given as
\begin{align}
\label{eq:def_power_function}
P_{X_n}(x) :=&\ \Vert k(\cdot, x) - \Pi_{V(X_n)}(k(\cdot, x)) \Vert_{\ns} \nonumber\\
=&\ \sup_{0 \neq f \in \ns} \frac{|(f-\Pi_{V(X_n)}(f))(x)|}{\Vert f \Vert_{\ns}}.
\end{align}
Obviously it holds $P_{X_n}(x_i) = 0$ for all $i=1, \dots, n$, and the standard power function estimate bounds the interpolation error as 
\begin{align}
\label{eq:standard_power_func_estimate}
|(f-s_n)(x)| 
&\leq P_{X_n}(x) \cdot \Vert f - s_n \Vert_{\ns} \nonumber\\
&= P_{X_n}(x) \cdot \Vert r_n \Vert_{\ns}\;\;\forall x\in\Omega,
\end{align}
where we denoted the residual as $r_n := r_n(f):=f - s_n$. 

Observe that any worst-case error bound on $|(f-\Pi_{V(X_n)}(f))(x)|$ over the entire $\ns$ transfers to the same decay of the power function via the second equality in 
Eq.~\eqref{eq:def_power_function}. For the large class of translational invariant kernels, that we will introduce below and that includes the notable class of 
Radial Basis Function (RBF) kernels, it is possible to refine this error estimate by bounding the decay of the power function in terms of the fill distance
\begin{align*}
h_{X_n} := h_{X_n, \Omega} := \sup_{x \in \Omega} \min_{x_j \in X_n} \Vert x - x_j \Vert_2.
\end{align*}
Depending on certain properties of the kernel, one may obtain in this way both algebraic and exponential rates in terms of $h_{X_n}$. Especially in 
the case of kernels whose RKHS is norm equivalent to a Sobolev space, these algebraic rates are provably quasi-optimal and may even be extended to certain 
functions that are outside of $\ns$ (see \cite{Narcowich2006}). %

These results are nevertheless bounded by the dependence on the filling of the space and by the independence on the target function $f$. Namely, the fill 
distance 
is at most decaying as $h_{X_n, \Omega} \asymp c_\Omega n^{-1/d}$ for {\em{quasi-uniform}} points, which are space-filling and target-independent.
On the other hand, a global target-dependent optimization of the interpolation points is a combinatorial and practically infeasible task, and thus approximated 
strategies have been proposed, and in particular greedy algorithms. 

Greedy algorithms in general are studied in various branches of mathematics and we point to 
\cite{Temlyakov2008} for a general treatment of their use in approximation. 
In kernel interpolation, a greedy algorithm starts with 
the empty set $X_0 := \emptyset$ and adds points incrementally as 
$X_{n+1} := X_n \cup \{ x_{n+1} \}$ according to some selection criterion $\eta^{(n)}$, that is
\begin{align*}
x_{n+1} := \argmax_{x \in \Omega \setminus X_n} \eta^{(n)}(x).
\end{align*}
Commonly used selection criteria in the greedy kernel literature are the $P$-greedy \cite{DeMarchi2005}, $f\cdot P$-greedy\footnote{We remark that we use the 
notation $f \cdot P$-greedy algorithm here because it fits 
better to our notation, while in the original publication it was called {\emph{psr}}-greedy (power scaled greedy).} \cite{Dutta2020}, $f$-greedy 
\cite{SchWen2000}, and $f/P$-greedy \cite{Mueller2009} criteria, and they choose the next point according to the following strategies. From now on
we use the short-hand notation $P_n(\cdot) := P_{X_n}(\cdot)$ whenever the power function is determined by some greedy algorithm. 
\begin{enumerate}[label=\roman*.]
\item $P$-greedy: \hspace{7mm} $\eta_P^{(n)}(x) = P_{n}(x)$,
\item $f \cdot P$-greedy: \hspace{2.2mm} $\eta_{f \cdot P}^{(n)}(x) = |r_n(x)| \cdot P_{n}(x)$,
\item $f$-greedy: \hspace{7.9mm} $\eta_f^{(n)}(x) = |r_n(x)|$,
\item $f/P$-greedy: \hspace{3mm} $\eta_{f/P}^{(n)}(x) = |r_n(x)|/P_{n}(x)$.
\end{enumerate}

These algorithms have been used in a series of applications (see e.g. 
\cite{Koeppl2018,Koeppel2018,Dutta2020,SchWen2000,Pazouki2011,Schaback2000b,Schmidt2018,HS2017a,Santin2021}), and overwhelming numerical evidence 
points to the fact that criteria which incorporate a residual-dependent term provide faster convergence, even if sometimes at the price of stability (see 
\cite{WENZEL2021105508} for a discussion of this fact for $f/P$-greedy, and \cite{Dutta2020} for $f\cdot P$-greedy).

The faster convergence is fully understandable since function adaptivity should clearly be beneficial to convergence speed. Nevertheless, the theoretical 
results are of opposite nature. Namely, for the $P$-greedy algorithm it is possible to prove quasi-optimality statements (see \cite{SH16b}), namely that 
whatever is the best known decay rate of the power function for arbitrarily optimized points, this transfers to the same decay of the power function associated 
to the points selected by $P$-greedy. Especially in the case of Sobolev spaces, these results can be proven to be optimal \cite{WENZEL2021105508}. On the other 
hand, the convergence theory for the target data-dependent algorithms is much weaker: The known results (see Section 
\ref{sec:background} for a detailed account) provide convergence of order at most $n^{-1/2}$, which is 
generally not only largely missing the practical observations, but also slower than 
the rates proven for $P$-greedy.

We remark that existing techniques to prove convergence of greedy algorithms in general Hilbert spaces are not directly transferable to this setting. 
Indeed, the first results on similar algorithms have been obtained in Matching Pursuit, and they work for finite dimensional spaces 
\cite{Mallat1993,Davis1997}. When transferred to the kernel setting (see \cite{SchWen2000}) these require a norm equivalence between the $\ns$- and the 
$\infty$-norm, which hold only for finite $n$. Subsequent general results on greedy algorithms (see \cite{DeVore1996}) require special assumptions on the 
target 
function, and the resulting rates are only of order $n^{-1/2}$. Another common strategy in the greedy 
literature makes use of the Restricted Isometry Property (see e.g. \cite{Cohen2017a}), which in the kernel setting translates to the requirement that the 
smallest eigenvalue $\lambda_n$ of the kernel matrix is bounded away from zero uniformly in $n$. This is not the case here, since it is known that 
$\lambda_n\leq \min_{1\leq j\leq n} P_{X_n\setminus \{x_j\}}(x_j)^2$ (see \cite{Schaback1995}), and we will see later that a fast convergence to zero of the 
right hand side of this inequality is the key of our analysis.
Especially, all these results prove convergence in the Hilbert space norm, which is generally too strong (to obtain convergence rates) since the interpolation operator is an orthogonal 
projection in $\ns$. 
We work instead with the $\infty$-norm, which allows to derive fast convergence, even if it introduces an additional difficulty since the norm of the error is 
not monotonically decreasing. 

The paper is organized as follows. 
After recalling additional facts on kernel greedy interpolation in Section \ref{sec:background}, 
we derive a new analysis of general greedy algorithms in general Hilbert spaces based on \cite{DeVore2013} (Section 
\ref{sec:analysis_abstract}). 

In Section \ref{sec:analysis_kernel} we frame the four selection rules into a joint scale of greedy algorithms by introducing $\beta$-greedy 
algorithms (Definition 
\ref{def:beta_greedy_algorithms}) which include $P$-greedy ($\beta=0$), $f\cdot P$-greedy ($\beta=1/2$), $f$-greedy ($\beta=1$), and $f/P$-greedy 
($\beta=\infty$), and we study them within a novel error analysis. 

These results are combined in Section \ref{sec:conv_rates} to obtain precise convergence rates of the minimum error $e_{\min}(f, n) := \min_{1\leq i\leq n} 
\left\|f - s_i\right\|_\infty$ . This measure allows us to circumvent the non monotonicity of the error, and we remark in particular that $e_{\min}(f, n)< 
\varepsilon$ for some $\varepsilon>0$ means that an error smaller than $\varepsilon$ is achieved using {\em{at most}} $n$ points.
As an exemplary result, we mention here the case where the rate of worst-case convergence in $\ns$ for a fixed set of $n$ 
interpolation points is $n^{-\alpha}$ for a given $\alpha>0$. In this case, 
for 
$\beta\in [0,1]$ we get new convergence rates of the form
\begin{align*}
e_{\min}(f, n) \leq c \log(n)^\alpha n^{-\beta/2} n^{-\alpha},\;\;n\geq n_0\in\N,
\end{align*}
with $c>0$. These results prove in particular that the worst case decay of the error that can be obtained in $\ns$ with a fixed 
sequence of points transfers to the $\beta$-greedy algorithms with an additional multiplicative factor of $\log(n)^\alpha n^{-\beta/2}$. Namely, adaptively 
selected points provide faster convergence than any fixed set of points.

Finally, Section \ref{sec:numerical_experiments} illustrates the results with analytical and numerical examples while the final Section 
\ref{sec:conclusion_outlook} presents the conclusion and gives an outlook.

\section{Background results on kernel interpolation}\label{sec:background}
We recall additional required background information on kernel based approximation and in particular greedy kernel 
interpolation. For a more detailed overview we refer the reader to \cite{Fasshauer2007, Fasshauer2015, Wendland2005}. 
We remark that in this section no special attention is paid to the occurring constants, which can change from line to line.

\subsection{Interpolation by translational invariant kernels}\label{subsec:kernel_interpolation_t_inv}

In many applications of interest the domain is a subset of the Euclidean space, i.e., $\Omega \subset \R^d$. In this case, a special kind of kernels is 
given by {\emph{translational invariant kernels}}, i.e., there exists a function $\Phi: \R^d \rightarrow \R$ with a continuous Fourier transform $\hat{\Phi}$ 
such that 
\begin{align*}
k(x, y) = \Phi(x - y) ~ \text{ for all } ~ x, y \in \R^d.
\end{align*}
We remark that the well known radial basis function kernels are a particular instance of translational invariant kernels.

Depending on the decay of the Fourier transform of the function $\Phi$, two classes of translational invariant kernels can be distinguished: 
\begin{enumerate}
\item We call the kernel $k$ a kernel of {\emph{finite smoothness}} $\tau > d/2$, if there exist constants $c_\Phi, C_\Phi > 0$ such that
\begin{align*}
c_\Phi (1+\Vert \omega \Vert_2^2)^{-\tau} \leq \hat{\Phi}(\omega) \leq C_\Phi (1 + \Vert \omega \Vert_2^2)^{-\tau}.
\end{align*}
The assumption $\tau > d/2$ is required in order to have a Sobolev embedding in $C^0(\Omega)$.
\item If the Fourier transform $\hat{\Phi}$ decays faster than at any polynomial rate, the kernel is called {\emph{infinitely smooth}}. 
\end{enumerate}

As mentioned in Section \ref{sec:introduction}, for these two types of kernels it is possible to derive error estimates by bounding the decay of the power 
function in terms of the fill distance. We have the following:
\begin{enumerate}
\item For kernels of finite smoothness $\tau > d/2$, given appropriate conditions on the domain $\Omega \subset \R^d$ (e.g.\ Lipschitz boundary and interior 
cone condition), the native space $\ns$ is norm equivalent to the Sobolev space $W_2^\tau(\Omega)$. By making use of this connection, error estimates for 
kernel interpolation can be obtained by using Sobolev bounds \cite{WendlandRieger2005, Narcowich2004} that give
\begin{align} \label{eq:error_estimate_algebraic}
\Vert P_{X_n} \Vert_{L^\infty(\Omega)} \leq \hat{c}_1 h_{X_n}^{\tau - d/2}. 
\end{align}
\item For kernels of infinite smoothness such as the Gaussian, the Multiquadric or the Inverse Multiquadric, we have
\begin{align} \label{eq:error_estimate_spectral}
\Vert P_{X_n} \Vert_{L^\infty(\Omega)} \leq \hat{c}_2 \exp(-\hat{c}_3 h_{X_n}^{-1}),
\end{align}
if the domain $\Omega$ is a cube. We remark that these error estimates are not limited to these three exemplary kernels. We point to \cite[Theorem 
11.22]{Wendland2005} which states a sufficient condition in order to obtain these exponential kind of error estimates.
\end{enumerate}

By looking at well distributed points such that $h_{X_n, \Omega} \leq c_\Omega n^{-1/d}$, these bounds from Eq.~\eqref{eq:error_estimate_algebraic} and 
\eqref{eq:error_estimate_spectral} can be cast only in terms of the number of interpolation points $n$, i.e.\
\begin{align}
\begin{aligned} \label{eq:error_estimates_in_n}
\Vert P_{X_n} \Vert_{L^\infty(\Omega)} &\leq \tilde{c}_1 n^{1/2-\tau/d}, \\
\Vert P_{X_n} \Vert_{L^\infty(\Omega)} &\leq \tilde{c}_2 \exp(-\tilde{c}_3 n^{1/d}).
\end{aligned}
\end{align}

\subsection{Greedy kernel interpolation} \label{subsec:greedy_kernel_interpol}
We collect the motivation, a few properties, and the existing analysis of the four selection criteria introduced in Section \ref{sec:introduction}:
\begin{enumerate}[label=\roman*.]
\item $P$-greedy: The $P$-greedy algorithm is the best analyzed one of the four algorithms named above. It aims at minimizing the error for all functions in 
the native space simultaneously, which is done by greedily minimizing the upper error bound from Eq. \eqref{eq:standard_power_func_estimate}, which is the 
power function. Thus, the selection criterion of the $P$-greedy algorithm is target data independent. For the $P$-greedy algorithm it holds $P_{n}(x_{n+1}) = 
\Vert P_{n} \Vert_{L^\infty(\Omega)}$. Several results on the $P$-greedy algorithm have been derived in \cite{SH16b, WENZEL2021105508}:
\begin{enumerate}
\item Corollary 2.2. in \cite{SH16b} showed convergence statements for the maximal power function value $\Vert P_n \Vert_{L^\infty(\Omega)}$ for radial basis 
function 
kernels, when $\Omega \subset \R^d$ has a Lipschitz boundary and satisfies an interior cone condition. It states
\begin{alignat*}{2}
\Vert P_{n} \Vert_{L^\infty(\Omega)} &\leq c_1 \cdot n^{1/2-\tau/d} \qquad &&\text{(finite smoothness} ~ \tau > d/2) \\
\Vert P_{n} \Vert_{L^\infty(\Omega)} &\leq c_2 \exp(- c_3 n^{1/d}) \qquad &&\text{(infinite smoothess)}.
\end{alignat*}
Via the standard power function bound from Eq.\ \eqref{eq:standard_power_func_estimate} these bounds directly give bounds on the approximation $\Vert f - s_n 
\Vert_{L^\infty(\Omega)}$. A few more details of the proof strategy of \cite{SH16b} will be recalled in Section \ref{sec:analysis_abstract}.
\item The paper \cite{WENZEL2021105508} showed further results for the case of kernels of finite smoothness $\tau > d/2$: Theorem 12 in \cite{WENZEL2021105508} 
showed 
that the decay rate on $\Vert P_n \Vert_{L^\infty(\Omega)}$ is sharp. The sequence of Theorems 15, 19 and 20 of \cite{WENZEL2021105508} further established 
that 
the resulting sequence of points are asymptotically uniformly distributed under some mild conditions. These results implied (optimal) stability statements in 
\cite[Corollary 22]{WENZEL2021105508}.
\end{enumerate}
\item $f$-greedy: The $f$-greedy algorithm aims at directly minimizing the residual by setting the currently largest residual to zero by introducing the next 
interpolation point at this point, i.e. it holds $|(f-s_n)(x_{n+1})| = \Vert f - s_n \Vert_{L^\infty(\Omega)}$. Existing results prove convergence of order 
$n^{-\ell/d}$ for kernels $k\in C^{2\ell}(\Omega\times\Omega)$ in $d=1$ (see Section 3.4 in \cite{Mueller2009}), while for general $d$ limited results are 
known, e.g. \cite[Korollar 3.3.8]{Mueller2009} states that 
\begin{align*}
\min_{j=1,\dots,n} \Vert f - s_j \Vert_{L^\infty(\Omega)} \leq C n^{-1/d}
\end{align*}
if $k\in C^2(\Omega\times \Omega)$. As mentioned before, these convergence results do not reflect the approximation speed of $f$-greedy that can be 
observed in numerical investigations. Additionally, in \cite{santin2020sampling} convergence of order $n^{-1/2}$ of the $\ns$-norm of the error is proven, but 
only under additional assumptions on $f$.
\item $f/P$-greedy: The $f/P$-greedy selection aims at minimizing the native space error of the residual as much as possible as it can be seen from Eq.\ 
\eqref{eq:ns_norm_via_sum}. We remark as a technical detail that the supremum of $|(f-s_n)(x)| / P_n(x)$ over $x \in \Omega \setminus X_n$ need not be attained 
as it was exemplified in Example 6 of \cite{WENZEL2021105508}. However this can be alleviated by choosing the next point $x_{n+1}$ such that 
$\frac{|r_n(x_{n+1})|}{P_n(x_{n+1})} \geq  (1-\epsilon) \cdot \sup_{x \in \Omega \setminus X_n} \frac{|r_n(x)|}{P_n(x)}$ for any $0 < \epsilon \ll 1$. As a 
convergence result, \cite[Theorem 3]{Wirtz2013} states 
\begin{align*}
\Vert f - s_n \Vert_{\ns} \leq C n^{-1/2},
\end{align*}
which however only holds for a quite restricted set of functions $f$, which has slightly been extended in \cite{santin2020sampling}.
\item $f \cdot P$-greedy: The idea of the just recently introduced $f \cdot P$-greedy algorithm is to have a combination of the power function dependence and 
the target data dependence in order to balance between the stability of the $P$-greedy algorithm and the target data dependence of the $f$-greedy algorithm. 
No convergence results were given in the original publication \cite{Dutta2020}. 
\end{enumerate}

In addition to the selection criteria, we remark that for a practical numerical implementation the greedy algorithms 
stop if a predefined bound (either on e.g. the accuracy or the numerical stability) is reached, or if the interpolant is exact.

Finally, to analyze and implement these algorithms it is useful to consider the {\emph{Newton basis}} $\{ v_j \}_{j=1}^n$ of $V_n$ (see 
\cite{Mueller2009b,Pazouki2011}), which is obtained by applying the Gram-Schmidt orthonormalization process to $\{k(\cdot, x_j), j=1, \dots, n \}$ whereby 
$\{x_j, j =1, \dots, n \}$ are the pairwise distinct points that are incrementally selected by the greedy procedure. We recall that we have
\begin{align*}
s_n(x) = \sum_{j=1}^n \langle f, v_j \rangle_{\ns} v_j(x),
\end{align*}
and it can be shown that it holds $\langle f, v_j \rangle_{\ns} = |(f-s_{j-1})(x_{j})| / P_{X_{j-1}}(x_{j})$. If $s_n \stackrel{n \rightarrow 
\infty}{\longrightarrow} f$ in $\ns$, we have
\begin{align}
\label{eq:ns_norm_via_sum}
\Vert f \Vert_{\ns}^2 = \sum_{j=1}^\infty \left( \frac{|(f-s_j)(x_{j+1})|}{P_{X_j}(x_{j+1})} \right)^2.
\end{align}

\section{Analysis of greedy algorithms in an abstract setting} \label{sec:analysis_abstract}

This section extends the abstract analysis of greedy algorithms in Hilbert spaces introduced in \cite{DeVore2013}. For this, let $\mathcal{H}$ be a 
Hilbert space with norm $\Vert \cdot \Vert = \Vert \cdot \Vert_\mathcal{H}$. Let $\mathcal{F} \subset \mathcal{H}$ be a compact subset and assume for 
notational convenience only that it holds $\Vert f \Vert_\mathcal{H} \leq 1$ for all $f \in \mathcal{F}$. 

We consider greedy algorithms that select elements $f_0, f_1, \dots$, without yet specifying any particular selection criterion. We define $V_n := 
\text{span}\{f_0, \dots, f_{n-1}\}$ and the following quantities, whereby $Y_n$ is any $n$-dimensional subspace of $\mathcal{H}$:
\begin{align} 
\begin{aligned} \label{eq:quantities}
d_n :=& d_n(\mathcal{F})_\mathcal{H} := \inf_{Y_n \subset \mathcal{H}} \sup_{f \in \mathcal{F}} \mathrm{dist}(f, Y_n)_\mathcal{H} \\
\sigma_n :=& \sigma_n(\mathcal{F})_\mathcal{H} := \sup_{f \in \mathcal{F}} \mathrm{dist}(f, V_n)_\mathcal{H} \\
\nu_n :=& \text{dist}(f_n, V_n)_\mathcal{H}.
\end{aligned}
\end{align}
The quantities $d_n$ and $\sigma_n$ have already been used in \cite{DeVore2013}, where $d_n$ is the Kolmogorov $n$-width of $\mathcal{F}$, and we recall that 
the compactness of $\mathcal F$ is equivalent to require that $\lim_n d_n = 0$ (see \cite{Pinkus1985}).
On the other hand, the newly introduced quantity $\nu_n$ does not seem in itself to be an interesting quantity for the abstract setting, and it has not been 
considered before. However, it will be the key quantity for our new analysis in the kernel setting in Section \ref{sec:analysis_kernel} and Section
\ref{sec:conv_rates}.

As we focus on Hilbert spaces, expressions like $\dist(f, V_n)$ can be computed via the orthogonal projector in $\mathcal{H}$ onto $V_n$, that we denote as 
$\Pi_{V_n}$. We have the following elementary properties:
\begin{enumerate}
\item Estimates: $d_n \leq \sigma_n$ and $\nu_n \leq \sigma_n$ for all $n \in \N$.
\item Monotonicity: $(\sigma_n)_{n \in \N}$ and $(d_n)_{n \in \N}$ are monotonically decreasing. 
\item Initial value: $d_0 \leq \sigma_0 \leq 1$.
\end{enumerate}

The paper \cite{DeVore2013} considers {\emph{weak}} greedy algorithms that choose, for some fixed $0 < \gamma \leq 1$, the elements 
$f_n$ such that
\begin{align} \label{eq:weak_criterion}
\nu_n \equiv \text{dist}(f_n, V_n)_\mathcal{H} \equiv \sigma_n(\{f_n\})_\mathcal{H} \geq \gamma \cdot \sup_{f \in \mathcal{F}} \sigma_n(\{f\})_\mathcal{H} = 
\gamma \cdot \sigma_n(\mathcal{F}),
\end{align}
and shows that, roughly speaking, an asymptotic polynomial or exponential decay of $d_n$ yields a polynomial or exponential decay of $\sigma_n$, 
i.e., the weak greedy algorithms essentially realize the Kolmogorov widths up to multiplicative constants. We remark that this analysis includes the
{\em strong} greedy algorithm, i.e., $\gamma=1$.

In the following we show in Subsection \ref{subsec:first_extension} that even without using the selection of Eq.~\eqref{eq:weak_criterion} -- 
i.e., the elements $f_0, f_1, \dots$ may even be randomly chosen within $\mathcal{F}$ -- comparable statements hold for $\nu_n$.

\subsection{Greedy approximation with arbitrary selection rules} \label{subsec:first_extension}
We start by stating a simple modification of \cite[Theorem 3.2.]{DeVore2013} and a subsequent corollary.

\begin{theorem} \label{th:abstract_estimate}
Consider a compact set $\mathcal{F}$ in a Hilbert space $\mathcal{H}$, and a greedy algorithm that selects elements from $\mathcal F$ with an arbitrary 
selection rule. 

We have the following inequalities between $\nu_n, 
\sigma_n$ and $d_n$ for any $N \geq 0, K \geq 1, 1 \leq m < K$:
\begin{align*}
\prod_{i=1}^K \nu_{N+i}^2 \leq \left( \frac{K}{m} \right)^m \left( \frac{K}{K-m} \right)^{K-m} \sigma_{N+1}^{2m} d_m^{2K-2m}.
\end{align*}
\end{theorem}
\begin{proof}
The result is obtained by simply omitting the last step in the proof of Theorem 3.2 in \cite{DeVore2013}.
Namely, we follow the original proof up to right before Eq.~(3.4), i.e., up to the bound on the quantity $a_{N+i, N+i}^2$. Using the second-to-last 
equation on p.~459 in \cite{DeVore2013} and our definition of $\nu_n$, in our notation we have
\begin{align*}
a_{N+i, N+i}^2 = \Vert f_{N+i} - \Pi_{V_{N+i}} f_{N+i} \Vert_\mathcal{H}^2 = \mathrm{dist}(f_{N+i}, V_{N+i})_\mathcal{H}^2 = \nu_{N+i}^2,
\end{align*}
and this gives the result. In the original paper an additional step in Eq.~(3.4) is used to obtain a bound on $\sigma_n$ instead of $\nu_n$.
\end{proof}

Similarly to the approach used in \cite{DeVore2013}, in the following corollary we make suitable choices 
of $N, K, m$ to specialize the result to the case of algebraically or exponentially decaying Kolmogorov widths.

\begin{cor} \label{cor:decay_abstract_setting_prod} 
Under the assumptions of Theorem \ref{th:abstract_estimate} the following holds.
\begin{enumerate}[label=\roman*)]
\item If $d_n(\mathcal{F}) \leq C_0 n^{-\alpha}, n\geq1$, 
then it holds
\begin{align}\label{eq:estimate_nu_prod_alg_improved}
\left(\prod_{i=n+1}^{2 n} \nu_{i} \right)^{1/n} 
&\leq2^{\alpha+1/2} \tilde{C}_0 e^\alpha  \log(n)^\alpha n^{-\alpha}, \;\;n\geq 3,
\end{align}
with $\tilde C_0 := \max\{1, C_0\}$.

\item If $d_n(\mathcal{F}) \leq C_0 e^{-c_0 n^\alpha}, n=1,2,\dots$, then it holds
\begin{align} \label{eq:estimate_nu_prod_exp}
\left( \prod_{i=n+1}^{2n} \nu_i \right)^{1/n} \leq \sqrt{2 \tilde{C}_0} \cdot e^{-c_1 n^\alpha}, \;\;n\geq 2,
\end{align}
with $\tilde{C}_0 := \max\{1, C_0\}$ and $c_1 = 2^{-(2+\alpha)}c_0 < c_0$.
\end{enumerate}
\end{cor}

\begin{proof}
First of all we observe that for $1 \leq m < n$, we have $0 < x := m/n < 1$. Using $x^{-x}(1-x)^{x-1} \leq 2$ for $x \in (0,1)$ we obtain
\begin{align*}
\left[ \left( \frac{n}{m} \right)^m \left( \frac{n}{n-m} \right)^{n-m} \right]^{1/n} = x^{-x}(1-x)^{x-1} \leq 2. \\
\end{align*}
We use Theorem \ref{th:abstract_estimate} for $N=K=n$ and any $1 \leq m < n$, i.e.\ we have
\begin{align}\label{eq:first_bound_alg}
\prod_{i=1}^n \nu_{n+i}^2 &\leq \left( \frac{n}{m} \right)^m \left( \frac{n}{n-m} \right)^{n-m} \sigma_{n+1}^{2m} d_m^{2n-2m} \notag \\
\Rightarrow 
\left( \prod_{i=1}^n \nu_{n+i} \right)^{1/n} 
&\leq \left[ \left( \frac{n}{m} \right)^m \left( \frac{n}{n-m} \right)^{n-m} \right]^{1/2n} \sigma_{n+1}^{m/n} d_m^{(n-m)/n} \notag \\
&\leq \sqrt{2} \sigma_{n+1}^{m/n} d_m^{(n-m)/n} 
\leq \sqrt{2} \cdot d_m^{(n-m)/n},
\end{align}
where we took the $2n$-th root for the second line and used the monotonicity and boundedness of $(\sigma_n)_{n \in \N}$ in the last step, i.e. 
$\sigma_{n+1}^{m/n}\leq \sigma_{1} ^{m/n} \leq 1$. 

In order to prove the statements i) and ii), we conclude now in two different ways:

\begin{enumerate}[label=\roman*)]
\item For $n$ fixed we choose a fixed $0 < \omega \ll 1$ and define $m^* := \lceil \omega n \rceil \in \N$, i.e.\ $\omega n \leq m^* < \omega n + 1$. 
Using $d_n \leq 1$ , $d_n \leq \tilde{C}_0 n^{-\alpha}$ with $\tilde{C}_0 := \max\{1, C_0\}$, and since $d_n$ is non-increasing, we can estimate:
\begin{align*}
\left( \prod_{i=1}^n \nu_{n+i} \right)^{1/n}
&\leq \sqrt{2} \cdot d_{m^*}^{(n-m^{*})/n} \leq \sqrt{2} \cdot d_{\lceil \omega n \rceil}^{(n-\omega n - 1)/n} \\
&\leq \sqrt{2} \tilde{C}_0^{(1-\omega) - 1/n} \lceil\omega n\rceil^{-\alpha(1-\omega) + \alpha/n}  \\
&\leq \sqrt{2} \tilde{C}_0^{(1-\omega) - 1/n} (\omega n)^{-\alpha(1-\omega) + \alpha/n}  \\
&\leq \sqrt{2} \tilde{C}_0^{(1-\omega)} \omega^{-\alpha(1-\omega)} n^{-\alpha(1-\omega)} (n^{1/n})^ \alpha  \\
&\leq \sqrt{2} \tilde{C}_0^{(1-\omega)} \omega^{-\alpha(1-\omega)} n^{-\alpha(1-\omega)} 2^ \alpha.
\end{align*}
It follows that for each $\omega \in (0, 1)$ it holds that  
\begin{align}\label{eq:intermediate_algebraic}
\left(\prod_{i=n+1}^{2 n} \nu_{i} \right)^{1/n} 
&\leq2^{\alpha+1/2} \tilde{C}_0 n^{-\alpha}\ C(\omega, n),
\end{align}
with $C(\omega, n) = \tilde{C}_0^{-\omega} \omega^{-\alpha(1-\omega)} n^{\alpha\omega}$. For each $n$, the inequality holds in particular for an optimally 
chosen $\bar\omega:=\bar\omega(n)$ in $(0,1)$. To find a good candidate $\bar\omega$ we minimize the upper bound $\tilde 
C(\omega, n):= \omega^{-\alpha} n^{\alpha\omega}$, which satisfies $C(\omega, n) \leq \tilde C(\omega, n)$ since $\omega, \alpha\geq 0$ and $\tilde C_0\geq 1$. 
It holds 
\begin{align*}
\partial_\omega \tilde C(\omega, n) = n^{\alpha\omega} \alpha \omega^{-1-\alpha} (-1 + \omega \log(n)),
\end{align*}
which vanishes in $\bar \omega = 1/\log(n)$, and it is negative on the left of this value and positive on the right. It follows that if $\bar \omega \in 
(0, 1)$, i.e., $n\geq 3$, then we can choose the constant $C(\bar \omega, n)$ in Equation \eqref{eq:intermediate_algebraic}, which gives the statement since 
\begin{align*}
C(\bar \omega, n)
&\leq \tilde C(\bar \omega, n) = \log(n)^{\alpha} n^{\alpha/\log(n)} =  e^{\alpha} \log(n)^{\alpha}. 
\end{align*}

\item We pick $m = \lceil n/2 \rceil$ and make use of the assumed decay $d_n(\mathcal{F}) \leq \tilde{C}_0 e^{-c_0 n^\alpha}$ to estimate
\begin{align*}
\left( \prod_{i=1}^n \nu_{n+i} \right)^{1/n} &\leq \sqrt{2} \cdot d_m^{(n-m)/n} = \sqrt{2} \cdot d_{\lceil n/2 \rceil}^{(n-\lceil n/2 \rceil)/n} \\
&\leq \sqrt{2} \cdot C_0^{1/2} e^{-c_0/2 (n/2)^\alpha \cdot (1-1/n)} \\
&= \sqrt{2} \cdot C_0^{1/2} e^{-2^{-1-\alpha} c_0 n^\alpha \cdot (1-1/n)} \\
&\stackrel{n \geq 2}{\leq} \sqrt{2} \cdot C_0^{1/2} e^{-2^{-2-\alpha} c_0 n^\alpha}\\
&= \sqrt{2 C_0}\ e^{-c_1 n^\alpha},
\end{align*}
where $c_1:= 2^{-(2+\alpha)} c_0$, and this concludes the proof.
\end{enumerate}
\end{proof}

\begin{rem} \label{rem:minimize_constant}
Observe that the constant $\tilde C_0 2^{\alpha+1/2} e ^{\alpha} = \tilde C_0\sqrt{2} \left(2 e\right) ^{\alpha}$ in 
\eqref{eq:estimate_nu_prod_alg_improved} is significantly smaller than 
the one obtained in \cite{DeVore2013} for the algebraic rate, which is $C_0 2^{5\alpha + 1}$. However, we have here instead the logarithmic factor in $n$, even 
if we presume that it may be possible to remove it with a finer analysis.
\end{rem}

\section{Analysis of greedy algorithms in the kernel setting} \label{sec:analysis_kernel}
This section introduces and analyses $\beta$-greedy algorithms, that are a scale of greedy algorithms which generalize the $P$-, $f \cdot P$-, $f$- and 
$f/P$-greedy algorithms.

We work under the assumption 
\begin{align}
\label{eq:assumption_kernel}
\Vert P_0 \Vert_{L^{\infty}(\Omega)} = \sup_{x \in \Omega} \Vert k(\cdot, x) \Vert_{\ns} = \sup_{x \in \Omega} \sqrt{k(x,x)} \leq 1.
\end{align}

\subsection{A scale of greedy algorithms: $\beta$-greedy} \label{subsec:beta_greedy}

We start with the definition of $\beta$-greedy algorithms.

\begin{definition} \label{def:beta_greedy_algorithms}
A greedy kernel algorithm is called $\beta$-greedy algorithm with $\beta \in [0, \infty]$, if the next interpolation 
point is chosen as follows. 
\begin{enumerate}
\item For $\beta \in [0, \infty)$ according to
\begin{align}
\begin{aligned} \label{eq:beta_greedy_selection_criterion}
x_{n+1} = \argmax_{x \in \Omega \setminus X_n} &|(f-s_n)(x)|^\beta \cdot P_n(x)^{1-\beta} 
.
\end{aligned}
\end{align}
\item For $\beta = \infty$ according to the $f/P$-greedy algorithm.
\end{enumerate}
\end{definition}
As depicted in Figure~\ref{fig:beta_greedy_visualization}, for $\beta = 0$ this is the $P$-greedy algorithm, for $\beta = 1/2$ it is the $f \cdot 
P$-algorithm, and for $\beta = 1$ it is the $f$-greedy algorithm.
In the limit $\beta \rightarrow \infty$ it makes sense to define the algorithm to be the $f/P$-greedy algorithm
\footnote{In this framework, the $f/P$-greedy algorithm is a limit case. This can be seen as an explanation  why the 
$f/P$-greedy selection rule is sometimes not well defined, as discussed in Example 6 in \cite{WENZEL2021105508}.}. 

Observe that the $\beta$-greedy algorithms are well defined also for $1 < \beta < \infty$. Indeed, in this case $1 - \beta < 
0$ and thus the power function part occurs as a divisor, and this may potentially be a problem since $P_n(x_i) = 0$ for all $1\leq i \leq n$. Nevertheless,  
the standard power function estimate gives
\begin{align*}
|(f-s_n)(x)|^\beta \cdot P_n(x)^{1-\beta} &= \frac{P_n(x)^\beta \cdot \Vert f-s_n \Vert_{\ns}^\beta}{P_n(x)^{\beta-1}} \leq \Vert f-s_n \Vert_{\ns}^\beta 
P_n(x),
\end{align*}
i.e.\ it holds $\lim_{x \rightarrow x_j} |(f-s_n)(x)|^\beta \cdot P_n(x)^{1-\beta} = 0$ for all $x_j \in X_n$.

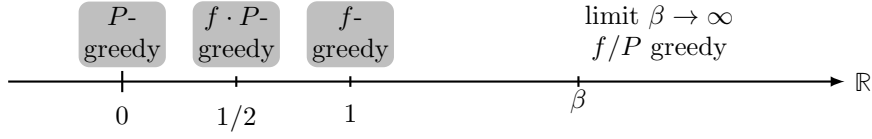
\begin{figure}[t]
\label{fig:number_line}
\setlength\fwidth{.4\textwidth}
\begin{center}
\begin{tikzpicture}[>=latex, thick]
\draw[->] (0,0) -- (11cm,0) node [right] {$\R$};
\draw (1.5,-4pt) -- (1.5,4pt) node[below=10pt]{0};
\draw (1.5,5pt) node[above, align=center, fill=lightgray, rounded corners, inner sep=2pt]{$P$- \\ greedy};
\draw (3,-2pt) -- (3,2pt);
\draw (3, 5pt) node[above, align=center, fill=lightgray, rounded corners, inner sep=2pt]{$f \cdot P$- \\ greedy} node[below=10pt]{1/2};
\draw (4.5,-2pt) -- (4.5,2pt);
\draw (4.5, 5pt) node[above, align=center, fill=lightgray, rounded corners, inner sep=2pt]{$f$- \\greedy} node[below=10pt]{1};
\draw (7.5,-2pt) -- (7.5,2pt);
\draw (7.5,5pt) node[above right, align=center, rounded corners, inner sep=2pt]{$\text{limit} ~ \beta \rightarrow \infty$  \\ $f/P$ greedy} node[below=1ex]{$\beta$};
\end{tikzpicture}
\end{center}
\vspace{-.7cm}
\caption{Visualization of the scale of the $\beta$-greedy algorithms on the real line. Several important cases for $\beta \in \{0, 1/2, 1\}$ and $\beta 
\rightarrow \infty$ are marked.}
\label{fig:beta_greedy_visualization}
\end{figure}

\begin{rem}[Generalizations of the $\beta$-selection rule]
We remark that it is sufficient to consider only one parameter $\beta > 0$ for the weighting of $|(f-s_n)(x)|$ and $P_n(x)$ as it was done in Eq.~
\eqref{eq:beta_greedy_selection_criterion}, in the sense that using two different parameters would be useless. 
Indeed, due to the strict monotonicity of the function $x \mapsto x^{1/\alpha}$, for some $\alpha > 0$ and for $\gamma 
\in \R$ it holds
\begin{align*}
\argmax_{x \in \Omega \setminus X_n} |(f-s_n)(x)|^\alpha \cdot P_n(x)^\gamma = \argmax_{x \in \Omega \setminus X_n} |(f-s_n)(x)| \cdot P_n(x)^{\gamma/\alpha},
\end{align*}
which shows that only the ratio $\gamma/\alpha$ is decisive. The specific parametrization via $\beta$ and $1-\beta$ in Eq. 
\eqref{eq:beta_greedy_selection_criterion} was chosen in order to obtain $f/P$-greedy as the limit case $\beta \rightarrow \infty$.
\end{rem}

\subsection{Analysis of $\beta$-greedy algorithms}
We can now prove the convergence of these algorithms.
So far, analysis of greedy kernel algorithms mainly focused on estimates on $\Vert f - s_i \Vert_{L^\infty(\Omega)}$. Here and in the following, different 
quantities will be analyzed with the goal of bounding instead $\min_{i=n+1, \dots, 2n} \Vert f - s_i \Vert_{L^\infty(\Omega)}$. 
We remark that no requirements on the kernel $k$ or the set $\Omega$ are needed for the results of this section, and especially for Theorem 
\ref{th:final_result}, as the proofs are based solely on RKHS theory.

We start by proving a key technical statement for greedy kernel interpolation that provides a bound on the product of the residual terms $r_i := f - s_i$. 
This result holds independently of the strategy that is used to select the points, greedy or not.

\begin{lemma} \label{lem:estimate_product}
For any sequence $\{ x_i \}_{i \in \N} \subset \Omega$ and any $f \in \ns$ it holds for all ${n=1, 2, \dots}$ that
\begin{align} \label{eq:estimate_product}
\left[ \prod_{i=n+1}^{2n} |r_i(x_{i+1})| \right]^{1/n} &\leq n^{-1/2} \cdot \Vert r_{n+1} \Vert_{\ns} \cdot \left[ \prod_{i=n+1}^{2n} P_i(x_{i+1}) 
\right]^{1/n}.
\end{align}
\end{lemma}

\begin{proof}
Let 
\begin{equation*}
R_n^2:= \left[ \prod_{i=n+1}^{2n} \left( \frac{r_i(x_{i+1})}{P_i(x_{i+1})} \right)^2 \right]^{1/n}.
\end{equation*}
The geometric arithmetic mean inequality gives
\begin{align*} 
R_n^2 
&\leq \frac{1}{n} \sum_{i=n+1}^{2n} \left(\frac{r_i(x_{i+1})}{P_i(x_{i+1})} \right)^2
= \frac{1}{n}\left(\sum_{i=1}^{2n} \left(\frac{r_i(x_{i+1})}{P_i(x_{i+1})} \right)^2 
- \sum_{i=1}^{n} \left(\frac{r_i(x_{i+1})}{P_i(x_{i+1})} \right)^2 \right).
\end{align*}
We now use Eq.\ \eqref{eq:ns_norm_via_sum} applied to $s_{2n+1}$ and $s_{n+1}$, and the properties of orthogonal projections to obtain
\begin{align*} 
R_n^2
&\leq \frac{1}{n} \left( \Vert s_{2n+1} \Vert_{\ns}^2 - \Vert s_{n+1}\Vert_{\ns}^2 \right)
\leq \frac{1}{n}  \left( \Vert f \Vert_{\ns}^2 - \Vert s_{n+1} \Vert_{\ns}^2 \right)\\
&= \frac{1}{n}  \Vert f - s_{n+1} \Vert_{\ns}^2
= \frac{1}{n}  \Vert r_{n+1} \Vert_{\ns}^2.
\end{align*}
It follows that $R_n \leq n^{-1/2} \cdot \Vert r_{n+1} \Vert_{\ns}$, and thus
\begin{equation*}
\left[ \prod_{i=n+1}^{2n} |r_i(x_{i+1})| \right]^{1/n} \leq n^{-1/2} \cdot \Vert r_{n+1} \Vert_{\ns} \cdot \left[ \prod_{i=n+1}^{2n} 
P_i(x_{i+1}) \right]^{1/n}.
\end{equation*}

\end{proof}

In order to derive convergence statements in the ${L^\infty(\Omega)}$ norm based on Lemma \ref{lem:estimate_product}, it is now required to 
find a relationship between $|r_i(x_{i+1})|$ and $\Vert r_i \Vert_{L^\infty(\Omega)}$. To this end, we have the following Lemma for $\beta$-greedy algorithms. 
Observe that the sequence of points depends on the value of $\beta$, i.e.\ $x_n \equiv x_n^{(\beta)}$, but for notational convenience we drop the superscript.

\begin{lemma} \label{lem:beta_greedy}
Any $\beta$-greedy algorithm with $\beta \in [0,\infty]$ applied to a function $f \in \ns$ satisfies for $i = 0, 1, \dots$:
\begin{enumerate}[label={\alph*)}]
\item In the case of $\beta \in [0, 1]$:
\begin{align}\label{eq:bound_ri_first}
\Vert r_i \Vert_{L^\infty(\Omega)} &\leq |r_i(x_{i+1})|^{\beta} \cdot P_i(x_{i+1})^{1-\beta} \cdot \Vert f - s_i \Vert_{\ns}^{1-\beta}.
\end{align}

\item In the case of $\beta \in (1, \infty]$ with $1/\infty := 0$:
\begin{align}\label{eq:bound_ri_second}
\Vert r_i \Vert_{L^\infty(\Omega)} &\leq \frac{|r_i(x_{i+1})|}{P_i(x_{i+1})^{1-1/\beta}} \cdot \Vert P_i \Vert_{L^\infty(\Omega)}^{1 - 1/\beta}.
\end{align}
\end{enumerate}
\end{lemma}

\begin{proof}
We prove the two cases separately:
\begin{enumerate}[label={\alph*)}]
\item For $\beta = 0$, i.e.\ the $P$-greedy algorithm, this is the standard power function estimate in conjunction with the $P$-greedy selection criterion 
$P_n(x_{n+1}) = \Vert P_n \Vert_{L^\infty(\Omega)}$. For $\beta = 1$ this holds with equality as it is simply the selection criterion 
of $f$-greedy since we have here $r_n(x_{n+1}) = \Vert r_n \Vert_{L^\infty(\Omega)}$. We thus consider $\beta \in (0, 1)$ and let $\tilde{x}_{i+1} \in 
\Omega$ be such that $|r_i(\tilde{x}_{i+1})| = \Vert r_i \Vert_{L^\infty(\Omega)}$. Then the selection criterion from 
Eq.~\eqref{eq:beta_greedy_selection_criterion} gives
\begin{equation*}
|r_i(x)|^\beta \cdot P_i(x)^{1-\beta} \leq |r_i(x_{i+1})|^\beta \cdot P_i(x_{i+1})^{1-\beta}\;\; \forall x\in\Omega,
\end{equation*}
and in particular
\begin{align*}
P_i(\tilde{x}_{i+1}) \leq \frac{|r_i(x_{i+1})|^{\frac{\beta}{1-\beta}}}{|r_i(\tilde{x}_{i+1})|^{\frac{\beta}{1-\beta}}} \cdot P_i(x_{i+1}).
\end{align*}
Using this bound with the standard power function estimate gives
\begin{align*}
\Vert r_i \Vert_{L^\infty(\Omega)} &= |r_i(\tilde{x}_{i+1})| \leq P_i(\tilde{x}_{i+1}) \cdot \Vert f - s_i \Vert_{\ns} \\
&\leq \frac{|r_i(x_{i+1})|^{\frac{\beta}{1-\beta}}}{|r_i(\tilde{x}_{i+1})|^{\frac{\beta}{1-\beta}}} \cdot P_i(x_{i+1}) \cdot \Vert f - s_i \Vert_{\ns} \\
&= \frac{|r_i(x_{i+1})|^{\frac{\beta}{1-\beta}}}{\Vert r_i \Vert_{L^\infty(\Omega)}^{\frac{\beta}{1-\beta}}} \cdot P_i(x_{i+1}) \cdot \Vert f - s_i \Vert_{\ns}.
\end{align*}
This can be rearranged for $\Vert r_i \Vert_{L^\infty(\Omega)}$ to yield the final result.
\item For $\beta \in (1, \infty)$, the selection criterion from Eq.~\eqref{eq:beta_greedy_selection_criterion} can be rearranged to
\begin{align*}
|r_i(x)|^\beta &\leq \frac{|r_i(x_{i+1})|^\beta}{P_i(x_{i+1})^{\beta - 1}} \cdot P_i(x)^{\beta - 1} \;\;\forall {x \in \Omega \setminus X_i},
\end{align*}
and taking the supremum $\sup_{x \in \Omega \setminus X_i}$ gives
\begin{align*}
\Vert r_i \Vert_{L^\infty(\Omega)} &\leq \frac{|r_i(x_{i+1})|}{P_i(x_{i+1})^{\frac{\beta - 1}{\beta}}} \cdot \Vert P_i 
\Vert_{L^\infty(\Omega)}^{\frac{\beta - 1}{\beta}}.
\end{align*}
For $\beta = \infty$, the selection criterion of the $f/P$-greedy algorithm can be directly rearranged to yield the statement (when using the notation 
$1/\infty = 0$).
\end{enumerate}
\end{proof}

Using the results of Lemma \ref{lem:beta_greedy} as lower bounds on $|r_i(x_{i+1})|$, it is now possible to control the left hand 
side of Inequality \eqref{eq:estimate_product}. This gives the main theorem of this section:

\begin{theorem} \label{th:final_result}
Any $\beta$-greedy algorithm with $\beta \in [0,\infty]$ applied to a function $f \in \ns$ satisfies the following error bound for $n=1, 2, \dots$:
\begin{enumerate}[label={\alph*)}]
\item In the case of $\beta \in [0, 1]$:
\begin{align} \label{eq:final_result_1}
\left[ \prod_{i=n+1}^{2n} \Vert r_i \Vert_{L^\infty(\Omega)} \right]^{1/n} \leq n^{-\beta/2} \cdot \Vert r_{n+1} \Vert_{\ns} \cdot \left[ 
\prod_{i=n+1}^{2n} P_i(x_{i+1}) \right]^{1/n}.
\end{align}

\item In the case of $\beta \in (1, \infty]$ with $1 / \infty := 0$:
\begin{equation}
\begin{aligned} \label{eq:final_result_2}
\left[ \prod_{i=n+1}^{2n} \Vert r_i \Vert_{L^\infty(\Omega)} \right]^{1/n} \leq n^{-1/2} &\cdot \Vert r_{n+1} \Vert_{\ns} \cdot \left[ \prod_{i=n+1}^{2n} 
P_i(x_{i+1})^{1/\beta} \right]^{1/n}. %
\end{aligned}
\end{equation}
\end{enumerate}
\end{theorem}

\begin{proof}
We prove the two cases separately:
\begin{enumerate}[label={\alph*)}]
\item 
For $\beta = 0$, i.e. $P$-greedy, Eq.~\eqref{eq:bound_ri_first} gives $\Vert r_i \Vert_{L^\infty(\Omega)} \leq P_i(x_{i+1}) \cdot \Vert r_i 
\Vert_{\ns}$. 
Taking the product $\prod_{i=n+1}^{2n}$ and the $n$-th root in conjunction with the estimate $\Vert r_i \Vert_{\ns} \leq \Vert r_{n+1} \Vert_{\ns}$ for $i = 
n+1, \dots, 2n$ gives the result. 

For $\beta \in (0, 1]$, we start by reorganizing the estimate \eqref{eq:bound_ri_first} of Lemma \ref{lem:beta_greedy} to get
\begin{equation*}
|r_i(x_{i+1})| \geq \left(\Vert r_i \Vert_{L^\infty(\Omega)}^{1/\beta} \right) / \left(P_i(x_{i+1})^{\frac{1-\beta}{\beta}} \cdot 
\Vert r_i \Vert_{\ns}^{\frac{1-\beta}{\beta}}\right),
\end{equation*}
and we use this to bound the left hand side of Eq.~\eqref{eq:estimate_product} as
\begin{align*} 
n^{-1/2} \cdot &\Vert r_{n+1} \Vert_{\ns} \cdot \left[ \prod_{i=n+1}^{2n} P_i(x_{i+1}) \right]^{1/n} \geq \left[ \prod_{i=n+1}^{2n} |r_i(x_{i+1})| 
\right]^{1/n} \\
&\geq \left[ \prod_{i=n+1}^{2n} \left(\Vert r_i \Vert_{L^\infty(\Omega)}^{1/\beta} \right) / \left(P_i(x_{i+1})^{\frac{1-\beta}{\beta}} \cdot 
\Vert r_i \Vert_{\ns}^{\frac{1-\beta}{\beta}}\right) \right]^{1/n} \\
&= \left[ \prod_{i=n+1}^{2n} \Vert r_i \Vert_{L^\infty(\Omega)}^{1/\beta} \right]^{1/n} \left[ \prod_{i=n+1}^{2n} P_i(x_{i+1})^{\frac{1-\beta}{\beta}} \cdot 
\Vert r_i \Vert_{\ns}^{\frac{1-\beta}{\beta}} \right]^{-1/n}.
\end{align*}
Rearranging the factors, and using again the fact that $\Vert r_i \Vert_{\ns} \leq \Vert r_{n+1} \Vert_{\ns}$ for $i = 
n+1, \dots, 2n$, gives
\begin{align*}
&\left[ \prod_{i=n+1}^{2n} \Vert r_i \Vert_{L^\infty(\Omega)}^{1/\beta} \right]^{1/n} \\
&\leq n^{-1/2} \cdot \Vert r_{n+1} \Vert_{\ns} \cdot \left[\prod_{i=n+1}^{2n} P_i(x_{i+1})^{1/\beta} \right]^{1/n} \cdot \left[ \prod_{i=n+1}^{2n} \Vert r_i 
\Vert_{\ns}^{\frac{1-\beta}{\beta}}  \right]^{1/n} \\
&\leq n^{-1/2} \cdot \Vert r_{n+1} \Vert_{\ns} \cdot \left[ \prod_{i=n+1}^{2n} P_i(x_{i+1})^{1/\beta} \right]^{1/n} \cdot \Vert r_{n+1} 
\Vert_{\ns}^{\frac{1-\beta}{\beta}} \\
&\leq n^{-1/2} \cdot \Vert r_{n+1} \Vert_{\ns}^{1/\beta} \cdot \left[ \prod_{i=n+1}^{2n} P_i(x_{i+1})^{1/\beta} \right]^{1/n}.
\end{align*}
Now, the inequality can be raised to the exponent $\beta$ to give the final statement.

\item For $\beta \in (1, \infty]$ we proceed similarly by first rewriting Eq.~\eqref{eq:bound_ri_second} of Lemma \ref{lem:beta_greedy} as
\begin{equation*}
|r_i(x_{i+1})| \geq \left(\Vert r_i \Vert_{L^\infty(\Omega)} \cdot P_i(x_{i+1})^{1-1/\beta}\right)/\left(\Vert P_i \Vert_{L^\infty(\Omega)}^{1-1/\beta}\right),
\end{equation*}
and we lower bound the left hand side of Equation \eqref{eq:estimate_product} as
\begin{align*} 
n^{-1/2} \cdot &\Vert r_{n+1} \Vert_{\ns} \cdot \left[ \prod_{i=n+1}^{2n} P_i(x_{i+1}) \right]^{1/n} \geq \left[ \prod_{i=n+1}^{2n} |r_i(x_{i+1})| 
\right]^{1/n} \\
&\geq \left[ \prod_{i=n+1}^{2n} \left(\Vert r_i \Vert_{L^\infty(\Omega)} \cdot P_i(x_{i+1})^{1-1/\beta}\right)/\left(\Vert P_i 
\Vert_{L^\infty(\Omega)}^{1-1/\beta}\right)\right]^{1/n}.
\end{align*}
Rearranging for $\left[ \prod_{i=n+1}^{2n} \Vert r_i \Vert_{L^\infty(\Omega)} \right]^{1/n}$ yields
\begin{align*}
&\left[ \prod_{i=n+1}^{2n} \Vert r_i \Vert_{L^\infty(\Omega)} \right]^{1/n} \\
&\leq n^{-1/2} \cdot \Vert r_{n+1} \Vert_{\ns} \cdot \left[\prod_{i=n+1}^{2n} \Vert P_i \Vert_{L^\infty(\Omega)}^{1-1/\beta} \right]^{1/n} \cdot \left[ 
\prod_{i=n+1}^{2n} P_i(x_{i+1})^{1/\beta} \right]^{1/n},
\end{align*}
which gives the final result due to $\Vert P_i \Vert_{L^\infty(\Omega)} \leq 1$ for all $i = 0, 1, ..$. 
\end{enumerate}
\end{proof}

\subsection{An improvement of the standard estimate}
As an additional consequence of Lemma \ref{lem:beta_greedy}, the following Corollary \ref{cor:improved_standard_power_func_estimate} gives a 
new inequality that can be seen as an \textit{improved standard power function estimate}, i.e. an improvement compared to the standard power 
function estimate from Eq.~\eqref{eq:standard_power_func_estimate}, that holds for any $\beta$-greedy algorithm.

\begin{cor}[Improved standard power function estimate] \label{cor:improved_standard_power_func_estimate}
Any $\beta$-greedy algorithm with $\beta \in [0,\infty]$ applied to a function $f \in \ns$ satisfies for $i = 0, 1, \dots$ the following improved standard 
power function estimate (with $1/\infty := 0$):
\begin{align} \label{eq:improved_standard_power_func_estimate}
\Vert r_i \Vert_{L^\infty(\Omega)} &\leq \Vert r_i \Vert_{\ns} \cdot 
\left\{
\begin{array}{ll}
P_i(x_{i+1}) & \beta \in [0, 1] \\
P_i(x_{i+1})^{1/\beta} & \beta \in (1, \infty] \\
\end{array}.
\right.
\end{align}
\end{cor}

\begin{proof}
For both $\beta \in [0, 1]$ and $\beta \in (1, \infty]$ we use the upper bounds on $\Vert r_i \Vert_{L^\infty(\Omega)}$ as stated in Lemma 
\ref{lem:beta_greedy} 
and further estimate the quantity $|r_i(x_{i+1})|$ via the standard power function estimate from Eq.\ \eqref{eq:standard_power_func_estimate} to get
\begin{align*}
\Vert r_i \Vert_{L^\infty(\Omega)} 
&\leq |r_i(x_{i+1})|^{\beta} \cdot P_i(x_{i+1})^{1-\beta} \cdot \Vert r_i \Vert_{\ns}^{1-\beta}
\leq P_i(x_{i+1}) \cdot \Vert r_i \Vert_{\ns}
\end{align*}
for $\beta \in [0,1]$, and
\begin{align*}
\Vert r_i \Vert_{L^\infty(\Omega)} 
&\leq \frac{|r_i(x_{i+1})|}{P_i(x_{i+1})^{1-1/\beta}} \cdot \Vert P_i \Vert_{L^\infty(\Omega)}^{1 - 1/\beta}
\leq P_i(x_{i+1})^{1/\beta} \cdot \Vert r_i \Vert_{\ns}
\end{align*}
for $\beta \in (1, \infty]$ by using $\Vert P_i \Vert_{L^\infty(\Omega)} \leq \Vert P_0 \Vert_{L^\infty(\Omega)} \leq 1$ for all $i = 0, 1, ..$ (see Eq.\ \eqref{eq:assumption_kernel}).
\end{proof}

The estimate from Eq.~\eqref{eq:improved_standard_power_func_estimate} is an improved estimate in comparison to Eq.~\eqref{eq:standard_power_func_estimate}, 
in that it provides a bound on $\Vert r_i \Vert_{L^\infty(\Omega)}$ instead of $|r_i(x_{i+1})|$, and this is a strictly larger quantity except that in the case 
of the $f$-greedy algorithm (i.e. $\beta=0$), where they coincide. Moreover, for $\beta \in [0, 1]$ the right hand side of the estimates of 
Eq.~\eqref{eq:standard_power_func_estimate} and \eqref{eq:improved_standard_power_func_estimate} coincide, while for $\beta>1$ this improvement comes at the 
price of a smaller exponent on the power function term, since $1/\beta<1$.

\begin{rem}
We will see in the following how to obtain convergence rates of the term $\min_{n+1\leq i\leq 2n} \Vert r_i \Vert_{L^{\infty}(\Omega)}$. 
From a practitioner point of view this kind of result might be unsatisfactory, as it is unclear which interpolant $s_i$ gives the best 
approximation. In this case it is possible to resort to the improved standard power function estimate of Corollary 
\ref{cor:improved_standard_power_func_estimate}: This inequality suggests to pick $s_{i^*}$ with $i^* := \argmin_{n+1\leq i \leq 2n} P_i(x_{i+1})$.
\end{rem}

\section{Convergence rates for greedy kernel interpolation}
\label{sec:conv_rates}

We can finally combine the abstract Hilbert space analysis from Section \ref{sec:analysis_abstract} and the greedy kernel interpolation analysis from Section 
\ref{sec:analysis_kernel}, and apply them to concrete classes of kernels. 

First of all, we recall a convenient connection that was established in \cite{SH16b} between the abstract analysis of \cite{DeVore2013} and kernel 
interpolation. We repeat it as we need to include also the extension of Section \ref{sec:analysis_abstract}, i.e., the new quantity $\nu_n$.
The goal is to frame the $\beta$-greedy algorithms as particular instances of the general greedy algorithm of Section \ref{sec:analysis_abstract}. In this view 
we choose $\mathcal{H} = \ns$ and $\mathcal{F} = \{k(\cdot, x), x \in \Omega\}$. The fact that this set is compact is implied by the decay to zero of its 
Kolmogorov width, that is equivalent to the existence of a sequence of points such that the associated power function converges to 
zero (see Eq.\ \eqref{eq:connection_2} below). This choice means that $f = k(\cdot, x) \in \mathcal{F}$ can be uniquely associated with an $x \in 
\Omega$ and 
vice versa. This yields a realization of the abstract greedy algorithm that produces an approximation set 
\begin{equation*}
V_n = \Sp\{f_0, \dots, f_{n-1} \} = \Sp\{ k(\cdot, x_i) ~ | ~ i=1, \dots, n\} = V(X_n),
\end{equation*}
and thus this is a greedy kernel algorithm, with an appropriate selection rule. Table \ref{tab:connection} summarizes these assignments. 
\begin{table}[h!]
\centering
\caption{Connection between the abstract setting and the kernel setting.}
\label{tab:connection}
\begin{tabular}{|l|l|c|l|} 
 \hline
Abstract setting & $\mathcal{H}$ & $f \in \mathcal{F} $ & $\Sp \{ f_0, \dots, f_{n-1} \}$ \\ \hline
Kernel setting & $\ns$ & $ k(\cdot, x), ~ x \in \Omega$ & $\Sp \{ k(\cdot, x_i), x_i \in X_n \}$ \\ 
 \hline
\end{tabular}
\end{table}

With these choices, as can be seen from the definition in Eq.~\eqref{eq:quantities}, $\sigma_n$ is simply the maximal power function value and $\nu_n$ is the 
power function value at the selected point.  
\begin{align}\label{eq:connection_1}
\sigma_n & \equiv \sup_{f \in \mathcal{F}} \mathrm{dist}(f, V_n)_\mathcal{H} = \sup_{f \in \mathcal{F}} \Vert f - 
\Pi_{V_n}(f) \Vert_{\mathcal{H}} \nonumber\\
&~= \sup_{x \in \Omega} \Vert k(\cdot, x) - \Pi_{V(X_n)}(k(\cdot, x)) \Vert_{\ns} 
= \Vert P_n \Vert_{L^\infty(\Omega)},\\
\nu_n &\equiv \mathrm{dist}(f_n, V_n)_\mathcal{H} 
= \Vert f_n - \Pi_{V_n}(f) \Vert_{\mathcal{H}} \nonumber\\
&= \Vert k(\cdot, x_{n+1}) - \Pi_{V_n}(k(\cdot, x_{n+1}) \Vert_{\ns} = P_n(x_{n+1})\nonumber.
\end{align}
Moreover, $d_n$ can be similarly bounded as
\begin{align}\label{eq:connection_2}
d_n &\equiv \inf_{Y_n \subset \mathcal{H}} \sup_{f \in \mathcal{F}} \mathrm{dist}(f, Y_n)_\mathcal{H}  = \inf_{Y_n 
\subset \mathcal{H}} \sup_{f \in \mathcal{F}} \Vert f - \Pi_{Y_n}(f) \Vert_\mathcal{H}\\
&\leq \inf_{Y_n \subset \mathcal{F}} \sup_{f \in \mathcal{F}} \Vert f - \Pi_{Y_n}(f) \Vert_{\ns} = \inf_{X_n \subset \Omega} \Vert P_{X_n} 
\Vert_{L^\infty(\Omega)} \nonumber,
\end{align}
and thus any convergence statement on $\Vert P_{X_n} \Vert_{L^\infty(\Omega)}$ for a given set of points $X_n\subset\Omega$ gives via 
Eq.~\eqref{eq:connection_2} a bound on $d_n$. 

Additionally, observe that the assumption $\|f\|_\calh \leq 1$ for $f\in\mathcal F$ implies in the kernel setting that
\begin{equation}\label{eq:pf_normalization}
\Vert P_0 \Vert_{L^\infty(\Omega)} = \sup_{x\in\Omega} \sqrt{k(x,x)} = \sup_{x\in\Omega} \|k(\cdot,x)\|_{\ns}\leq 1.
\end{equation}

\subsection{Convergence rates for $\beta$-greedy algorithms} \label{subsec:beta_greedy_convergence}
From Theorem \ref{th:final_result} it is now easily possible to derive convergence statements and decay rates for the kernel greedy algorithms, by bounding the 
right-hand side by Inequality \eqref{cor:decay_abstract_setting_prod} and using the interpretations of $\nu_i$ and $d_n$ from from Eq.~\eqref{eq:connection_1} 
and Eq.~\eqref{eq:connection_2}. 

\begin{cor} \label{cor:decay_rates_greedy}
Assume that a $\beta$-greedy algorithm with $\beta \in [0,\infty]$ is applied to a function $f \in \ns$. Let $\alpha, C_0, c_0>0$ be given constants, and 
set $1 / \infty := 0$. 

\begin{enumerate}
\item If there exist a sequence $(X_n)_{n\in\N}\subset \Omega$ of sets of points such that
\begin{equation*}
\left\|f - \Pi_{X_n} f\right\|_{L^\infty(\Omega)} \leq C_0 n^{-\alpha} \|f\|_{\ns}\;\;\forall f\in \ns,
\end{equation*}
then for all $\beta\geq 0$ and for all $n\geq 3$ it holds
\begin{equation}\label{eq:decay_result_1}
\min_{n+1\leq i\leq 2n}\Vert r_i \Vert_{L^\infty(\Omega)} \leq C \cdot  n^{-\frac{\min\{1, \beta\}}{2}} (\log(n)\cdot n^{-1})^{\frac{\alpha}{\max\{1, \beta\}}} 
\Vert r_{n+1} \Vert_{\ns},
\end{equation}
with $C:=\left(2^{\alpha+1/2} \max\{1, C_0\} e^\alpha\right)^{\frac{1}{\max\{1, \beta\}}}$.
In particular
\begin{align*}
\min_{n+1\leq i\leq 2n} \Vert r_i \Vert_{L^\infty(\Omega)} &\leq C \cdot \log(n)^{\alpha} \cdot \Vert r_{n+1} \Vert_{\ns} \cdot \left\{
\begin{array}{ll}
n^{-\alpha-1/2} & f-\text{greedy} \\
n^{-\alpha-1/4} & f \cdot P-\text{greedy} \\
n^{-\alpha} & P-\text{greedy}
\end{array}
\right..
\end{align*}

\item If there exist a sequence $(X_n)_{n\in\N}\subset \Omega$ of sets of points\footnote{We remark that this sequence of sets of points does not need to be nested.} such that  
\begin{equation*}
\left\|f - \Pi_{X_n} f\right\|_{L^\infty(\Omega)} \leq C_0 e^{-c_0 n^\alpha} \|f\|_{\ns}\;\;\forall f\in \ns,
\end{equation*}
then for all $\beta\geq 0$ and for all $n\geq 2$ it holds
\begin{align}\label{eq:decay_result_2}
\min_{n+1\leq i\leq 2n}\Vert r_i \Vert_{L^\infty(\Omega)} 
&\leq C \cdot  n^{-\frac{\min\{1, \beta\}}{2}} e^{-c_1 n^\alpha} \Vert r_{n+1} \Vert_{\ns},
\end{align}
with $C:=\left(\sqrt{2\max\{1, C_0\}} \right)^{\frac{1}{\max\{1, \beta\}}}$ and $c_1 = 2^{-(2+\alpha)} c_0 / \max\{1, \beta\}$.
In particular
\begin{align*}
\min_{i=n+1, \dots, 2n} \Vert r_i \Vert_{L^\infty(\Omega)} \leq C \cdot e^{-c_1 n^\alpha} \cdot \Vert r_{n+1} \Vert_{\ns} \cdot \left\{
\begin{array}{ll}
n^{-1/2} & f-\text{greedy} \\
n^{-1/4} & f \cdot P-\text{greedy} \\
n^0 & P-\text{greedy}
\end{array}
\right..
\end{align*}

\item For $f/P$-greedy, for any kernel and for all $n\geq 1$ it holds 
\begin{align*}
\min_{n+1\leq i \leq 2n} \Vert r_i \Vert_{L^\infty(\Omega)} \leq n^{-1/2} \cdot \Vert r_{n+1} \Vert_{\ns}.
\end{align*}

\end{enumerate}
\end{cor}
\begin{proof}
The proof is a simple combination of Corollary \ref{cor:decay_abstract_setting_prod} and Theorem \ref{th:final_result}, with the addition of the following 
simple steps: \\
First, the worst case bounds in $\ns$ (either algebraic or exponential) imply the same bound on the power function via Eq.\ \eqref{eq:def_power_function}.
Second, in all cases we use the results of Theorem \ref{th:final_result} in combination with the bound
\begin{equation*}
\min_{i=n+1, \dots, 2n} \Vert r_i \Vert_{L^\infty(\Omega)}\leq \left[\prod_{i=n+1}^{2n} \| r_i\|_{L^\infty(\Omega)}\right]^{1/n}.
\end{equation*}

\noindent Then, Eq.\ \eqref{eq:final_result_1} and \eqref{eq:final_result_2} of Theorem \ref{th:final_result} can be jointly written as 
\begin{align*}
\left[\prod_{i=n+1}^{2n} \Vert r_i \Vert_{L^\infty(\Omega)} \right]^{1/n} 
\leq n^{-\frac{\min\{1, \beta\}}{2}} \cdot \Vert r_{n+1} \Vert_{\ns} \cdot \left[\prod_{i=n+1}^{2n} P_i(x_{i+1}) \right]^{\frac{1}{n\max\{1, \beta\}}}.
\end{align*}
Plugging the bounds of Corollary \ref{cor:decay_abstract_setting_prod} in the last inequality gives the result of the first two points. 
The third point directly follows from Eq.\ \eqref{eq:final_result_2} for $\beta = \infty$ due to $P_i(x_{i+1}) \leq 1$ for all $i=1, 2, ...$\ . %
\end{proof}

\subsection{Translational invariant kernels}
Strictly positive definite and translational invariant kernels are popular kernels for applications. To specialize our result to this interesting case, in this 
subsection we use the following assumption.

\begin{assumption} \label{ass:assumption}
Let $k(x,y) = \Phi(x-y)$ be a strictly positive definite translational invariant kernel with associated reproducing kernel Hilbert space $\ns$, whereby the 
domain $\Omega \subset \R^d$ is assumed to be bounded with Lipschitz boundary and interior cone condition.
\end{assumption}

In this context we have the following special case of Corollary \ref{cor:decay_rates_greedy}. To highlight the results in the most relevant cases, we state 
them only for $\beta \in \{0, 1/2, 1, \infty\}$ even if similar statements hold for general $\beta>0$.
\begin{cor} \label{cor:decay_rates_greedy_transl}
Under Assumptions \ref{ass:assumption}, any $\beta$-greedy algorithm with $\beta \in \{0, 1/2, 1, \infty\}$ applied to some function $f \in \ns$ satisfies the 
following error bounds for $n=0, 1, ..$, where the constants are defined as in Corollary \ref{cor:decay_rates_greedy}.

\begin{enumerate}
\item In the case of kernels of finite smoothness $\tau > d/2$
\begin{align*}
\min_{i=n+1, \dots, 2n} \Vert r_i \Vert_{L^\infty(\Omega)} &\leq C \cdot \log(n)^{\tau/d-1/2} \cdot \Vert r_{n+1} \Vert_{\ns} \cdot \left\{
\begin{array}{ll}
n^{-\tau/d} & f-\text{greedy} \\
n^{1/4-\tau/d} & f \cdot P-\text{greedy} \\
n^{1/2-\tau/d} & P-\text{greedy}
\end{array}
\right..
\end{align*}
\item In the case of kernels of infinite smoothness:
\begin{align*}
\min_{i=n+1, \dots, 2n} \Vert r_i \Vert_{L^\infty(\Omega)} \leq C \cdot e^{-c_1 n^{1/d}} \cdot \Vert r_{n+1} \Vert_{\ns} \cdot \left\{
\begin{array}{ll}
n^{-1/2} & f-\text{greedy} \\
n^{-1/4} & f \cdot P-\text{greedy} \\
n^0 & P-\text{greedy}
\end{array}
\right..
\end{align*}
\end{enumerate}
\end{cor}

Observe that for any $\beta \in (0, 1]$ we have the additional convergence of order $n^{-\beta/2}$ or $n^{-1/2}$ for $\beta > 1$. The 
additional decay is faster for increasing $\beta\in(0,1]$, i.e. increasing the weight of the target data-dependent term in the selection criterion gives 
better decay rates. Especially, the proven decay rate for $f$-greedy is better than the one for $f \cdot P$-greedy which is better than the one for $P$-greedy.

This additional convergence proves in particular that the Kolmogorov barrier can be broken, i.e., approximation rates that are better than the ones provided by 
the Kolmogorov width can be obtained for any function in $\ns$. Indeed, as discussed above any bound on $d_n$ turns into a bound on $\|P_n\|_{L^\infty(\Omega)}$, which 
can then be used in Corollary \ref{cor:decay_rates_greedy} or Corollary \ref{cor:decay_rates_greedy_transl}.

This is particularly relevant for the kernels whose RKHS is norm equivalent to a Sobolev space. But also other general kernels of low smoothness are of 
interest, since it might happen that the power function is decaying at arbitrarily slow speed, while the adaptive points selected by a $\beta$-greedy algorithm 
provide an additional convergence rate.

Moreover, the additional decay for $\beta> 0$ is dimension independent and thus it does not incur in the curse of dimensionality. This is of interest
in particular for translational invariant kernels of Corollary \ref{cor:decay_rates_greedy_transl}, as both the algebraic and the exponential decay of the 
power function (or Kolmogorov width) degrade with the dimension $d$ and thus the additional term gains more importance.

Despite this notable relevance, the estimates of Corollary \ref{cor:decay_rates_greedy} and Corollary \ref{cor:decay_rates_greedy_transl} are likely not 
optimal in the algebraic case. Indeed, for kernels with algebraically decaying Kolmogorov width, in the case of the $P$-greedy algorithm ($\beta=0$) bounds 
without the additional $\log(n)^{\alpha}$ factor are known \cite{SH16b}. We thus expect that the inconvenient additional $\log(n)^{\alpha}$ factor 
is not required for any of the $\beta$-greedy algorithms.
We remark that this factor is related to the the additional $\epsilon$ within Corollary \ref{cor:decay_abstract_setting_prod}, but we did not find a way to get 
rid of it, with exception of $\beta = 0$, i.e. the $P$-greedy case.
Moreover, we obtained our bounds by means of the worst-case bounds on $(\prod_{i=n+1}^{2n} P_i(x_{i+1}))^{1/n}$ from Corollary 
\ref{cor:decay_abstract_setting_prod}. Numerically, a faster decay than the worst case bound from Corollary \ref{cor:decay_abstract_setting_prod} can often be observed (see the examples in Subsection \ref{subsec:num_experiments_1}). Especially, for each $\beta$ value we obtain a different sequence of points and thus 
a different decay of the corresponding power function values.

\begin{rem}[Additional convergence orders]
Additional convergence orders can be obtained from the decay of $\Vert r_n \Vert_{\ns}$. Even if this quantity is in general decaying at arbitrarily slow speed 
for a general $f\in \ns$, we mention the case of superconvergence \cite{Schaback1999a,Schaback2018a}, which allows to 
bound $\Vert r_i \Vert_{\ns} \leq C_f \cdot \Vert P_i \Vert_{L^2(\Omega)}$ for special functions $f \in \ns$. The 
original superconvergence requirement $f = Tg$ (whereby $T$ is the kernel integral operator and $g 
\in L^2(\Omega)$, i.e.\ $(T g)(x) = \int_{\Omega} k(x,y) g(y) \mathrm{d}y$) can be extended to functions $f\in\ns$ such that 
$|\langle f, g \rangle_{\ns} | \leq C_f \cdot \Vert g \Vert_{L^q(\Omega)}$ for all $g \in \ns$ (see \cite[Theorem 
19]{santin2020sampling}).
\end{rem}

\begin{rem}[Stability]
The stability of the greedy interpolation, as computed here by the so-called direct method, is mainly linked to the smallest eigenvalue of the 
kernel interpolation matrix. A standard result \cite{Schaback1995} gives the upper estimate $\lambda_{\min}(X_n) \leq 
P_{n-1}(x_{n})^2$. In view of the estimates of Equation \eqref{eq:decay_result_1} and 
\eqref{eq:decay_result_2}, this means that a faster convergence based on a faster decay of the power function values $P_i(x_{i+1})$ directly negatively 
influences the stability. This holds especially for $\beta > 1$, because in this case the upper bound for the convergence in terms of the power function 
scales with the exponent $1/\beta < 1$.
\end{rem}

\begin{rem}[Other greedy selection rules]
The analysis above shows that $\gamma$-greedy algorithms which were introduced in \cite{WENZEL2021105508} are actually closer to the $P$-greedy 
algorithm than to target data-dependent algorithms for the case of kernels of finite smoothness $\tau > d/2$. In this case for $\gamma$-greedy algorithms
the decay of 
$P_n(x_{n+1})$ can be both lower and upper bounded by a constant times $n^{1/2-\tau/d}$. As the point selection criteria of $\gamma$-stabilized greedy 
algorithms first of all look at the power function value via $P_n(x_{n+1}) \geq \gamma \cdot \Vert P_n \Vert_{L^\infty(\Omega)}$, there is no relationship as 
in Eq.~\eqref{eq:beta_greedy_selection_criterion} ($\beta > 0$). Thus we cannot derive additional convergence rates.
\end{rem}

\begin{rem}[Other norms]
For kernels of finite smoothness $\tau > d/2$ on a set $\Omega$ with Lipschitz boundary satisfying an interior cone condition, the optimal rates of 
$L^p$-convergence are of order $\norm{L^p(\Omega)}{r_n} \leq c_p n^{-\tau/d+ (1/2 - 1/p)_+}$. This rate is matched by the $P$-greedy algorithm (see 
\cite[Corollary 22]{WENZEL2021105508}), since it is proven to select asymptotically uniformly distributed points. 

In the case of the $f$-greedy algorithm, we can use the additional factor $n^{-1/2}$ from Corollary \ref{cor:decay_rates_greedy_transl} to get rid of the 
conversion from the $L^p$ to the $L^\infty$ norm, i.e. we have
\begin{align*}
\norm{L^p(\Omega)}{r_n} \leq \meas(\Omega)^{1/p} \norm{L^\infty(\Omega)}{r_n} \leq c_\infty \log(n)^{\tau/d-1/2} n^{-\tau/d}.
\end{align*}
So we have almost $L^p$-optimal results (up to the poly-logarithmic factor) for $p \in [1, 2]$ and even improved convergence for $p \in (2, \infty]$. 
Similar statements hold for general $\beta$-greedy algorithms.
\end{rem}

\section{Examples} \label{sec:numerical_experiments}

\subsection{Visualization of results of abstract setting} \label{subsec:num_experiments_1}

This subsection visualizes the results from the abstract analysis in Section \ref{sec:analysis_abstract}, especially Subsection \ref{subsec:first_extension}. 
Again we make use of the links recalled in the beginning of Section \ref{sec:conv_rates}, especially in Eq.\ \eqref{eq:connection_1} and \eqref{eq:connection_2}. \\
We consider the domain $\Omega = [0, 1]^3 \subset \R^3$ and the Gaussian kernel with kernel width parameter $2$, i.e.\ $k(x,y) = \exp(-4 \Vert x-y \Vert_2^2)$. 
Four different sequences of points are considered, with colors refering to Figure \ref{fig:decay_power_value}: 
\begin{enumerate}[label=\roman*.]
\item Blue: $P$-greedy algorithm on the whole domain $\Omega$.
\item Violet: $P$-greedy algorithm on the subdomain $\Omega_2 := \{ x \in \Omega ~|~ (x)_3 = 1/2 \}$. Like this, the dimension is effectively reduced from $d=3$ 
to $d=2$.
\item Yellow, red: The points are randomly picked within $\Omega$.
\end{enumerate}

The results are displayed in Figure \ref{fig:decay_power_value}:
\begin{itemize}
\item The upper two figures displays the quantities $\sigma_n = \Vert P_n \Vert_{L^\infty(\Omega)}$ (left) and $\nu_n = P_n(x_{n+1})$ (right). 
\item The lower two figures display
\begin{alignat*}{2}
n &\mapsto \left( \prod_{j=n+1}^{2n} \sigma_j \right)^{1/n} = \left( \prod_{j=n+1}^{2n} \Vert P_j \Vert_{L^\infty(\Omega)} \right)^{1/n} \hspace{1cm} && \text{(left),} \\
n &\mapsto \left( \prod_{j=n+1}^{2n} \nu_j \right)^{1/n} = \left( \prod_{j=n+1}^{2n} P_j(x_{j+1}) \right)^{1/n} && \text{(right)}.
\end{alignat*}
\end{itemize}
For the numerical experiments, the domain $\Omega$ was discretized using $2 \cdot 10^4$ random points and $\Omega_2$ was discretized by projecting the random 
points related to $\Omega$ onto $\Omega_2$. The algorithms run until 300 points are selected or the next selected Power function value satisfies $P_n(x_{n+1}) < 
10^{-5}$. 

From the top left picture one can infer that the displayed quantity $\Vert P_n \Vert_{L^\infty(\Omega)}$ decays fastest for the $P$-greedy algorithm. This was 
expected, as the algorithms directly aims at minimizing this quantity. However the displayed quantity $\Vert P_n \Vert_{L^\infty(\Omega)}$ does not drop at all 
for the $P$-greedy algorithm on $\Omega_2$, as it picks only points from $\Omega_2$ and thus does not fill $\Omega$.

Contrarily, the top right picture shows that the displayed quantity $P_n(x_{n+1})$ decays faster for the $P$-greedy algorithm on $\Omega_2$, while for 
the $P$-greedy algorithm on $\Omega$ we have exactly the same curve due to $P_n(x_{n+1}) = \Vert P_n \Vert_{L^\infty(\Omega)}$. The two further point choices 
exhibit a wiggling, noisy behaviour on the displayed $P_n(x_{n+1})$ quantity, which is related to the random point choice.

The two lower figures refer to the geometric mean $(\Pi_{j=n+1}^{2n} ~ .. ~ )^{1/n}$ of the quantities of the upper figures. In the lower left figure, we can see that only the curve 
related to the $P$-greedy algorithm on $\Omega$ decays fast, the other curves do not decay at all or only slowly -- because the points are not chosen in a way 
to minimize the maximal Power function value $\Vert P_n \Vert_{L^\infty(\Omega)}$. Contrarily, the $P$-greedy algorithm on $\Omega$ exhibits the slowest 
decay of the quantity $(\Pi_{j=n+1}^{2n} \nu_j)^{1/n}$, which is the same curve as in the lower left figure due to $\nu_j = P_j(x_{j+1}) = \Vert P_j 
\Vert_{L^\infty(\Omega)} = \sigma_j$. However, all the three other choices of points provide a faster decay of the displayed quantity $(\prod_{j=n+1}^{2n} 
P_j(x_{j+1}))^{1/n} = (\prod_{j=n+1}^{2n} \nu_j)^{1/n}$ . The theoretical reason for (at least) the same decay as the $P$-greedy algorithm on $\Omega$ was proven in Corollary 
\ref{cor:decay_abstract_setting_prod}.

\begin{figure}[ht]
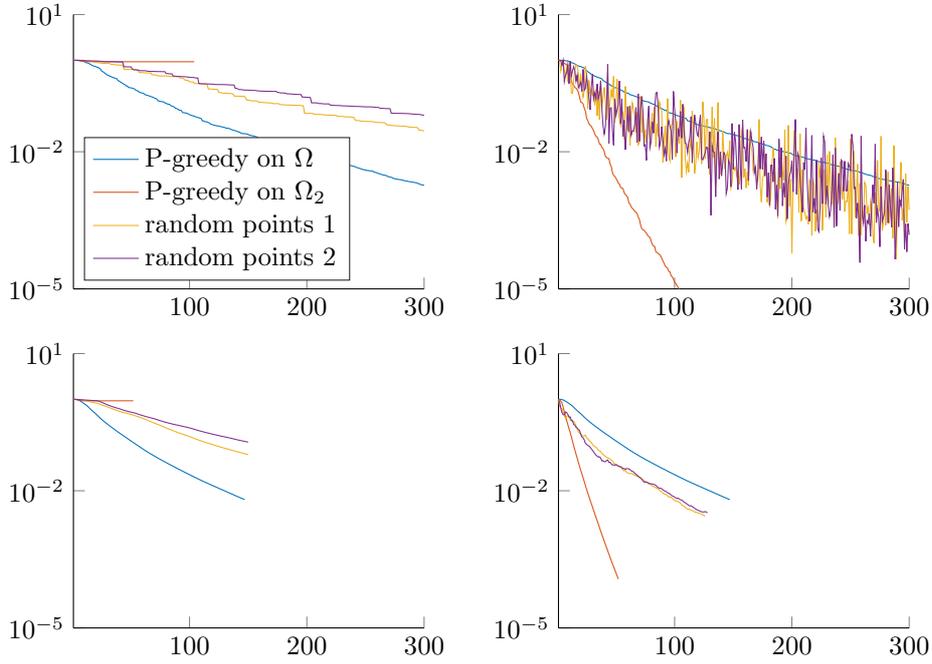

\setlength\fwidth{.4\textwidth}
\input{Figures/fig01_maximal_power_value.tex} ~
\input{Figures/fig02_selected_power_value.tex} \\
\input{Figures/fig03_product_maximal_power_value.tex} ~
\input{Figures/fig04_product_selected_power_value.tex} \\
\caption{Decay of several power function related quantities ($y$-axis) in the number of chosen interpolation points ($x$-axis) for four different choices of 
points: The upper two plots display the quantites $\sigma_n = \Vert P_n \Vert_{L^\infty(\Omega)}$ (left) respective $\nu_n = P_n(x_{n+1})$ (right). The lower 
two plots display the quantities $( \prod_{j=n+1}^{2n} \sigma_j )^{1/n} = ( \prod_{j=n+1}^{2n} \Vert P_j \Vert_{L^\infty(\Omega)} )^{1/n}$ (left) respective $( \prod_{j=n+1}^{2n} 
\nu_j )^{1/n} = ( \prod_{j=n+1}^{2n} P_j(x_{j+1}) )^{1/n}$ (right).}
\label{fig:decay_power_value}
\end{figure}

\subsection{$\beta$-greedy algorithms using the Wendland kernel} \label{subsec:num_experiments_2}

We consider the application of $\beta$-greedy algorithms for the particular example of the Wendland $k=0$ kernel on the domain $\Omega = [0,1]$, which is defined as
\begin{align*}
k(x,y) = \max(1 - |x-y|, 0),
\end{align*}
and thus a piecewise linear kernel. Its native space $\ns$ is norm equivalent to the Sobolev space $W^1_2(\Omega)$. It is immediate to see that kernel interpolation using the Wendland $k=0$ kernel on centers $X_n \subset \Omega$ boils down to linear spline interpolation on the subinterval $[\min X_n, \max X_n] \subset [0, 1]$. On $\Omega \setminus [\min X_n, \max X_n]$ the interpolant is still an affine function. \\
We consider the function $f: \Omega \rightarrow \R, 
x \mapsto x^\alpha$ for some $1/2 < \alpha < 1$. For $\alpha > 1/2$ it holds $f \in W^1_2(\Omega)$, thus $f \in \ns$.
It can be shown, that in the case of asymptotically uniform interpolation points --- i.e.\ $q_n \asymp h_n \asymp n^{-1}$ --- it is possible to lower-bound the error as (for details see Appendix \ref{sec:details_example})
\begin{align}
\label{eq:example_lower_bound_uniform}
\Vert f - s_n \Vert_{L^\infty(\Omega)} \geq C_\alpha \cdot n^{-\alpha}
\end{align}
for $C_\alpha > 0$. Furthermore, independent of the way the interpolation points $X_N$ were chosen (i.e.\ even for optimally chosen points), it holds
\begin{align}
\label{eq:example_lower_bound_optimal}
\Vert f - s_n \Vert_{L^\infty(\Omega)} \geq C \cdot n^{-2}
\end{align}
for some $C > 0$. Thus we can infer:
\begin{itemize}
\item Any (greedy) algorithm that yields asymptotically uniformly distributed points cannot have a convergence rate better than $n^{-\alpha}$ for this particular example. This includes especially the $P$-greedy algorithm, but also any $\gamma$-stabilized greedy algorithms \cite{WENZEL2021105508}, as they are known to provide asymptotically uniform points as well, see \cite[Theorem 20]{WENZEL2021105508}. Thus this example shows that $\gamma$-stabilized greedy algorithms cannot be expected in general to give a better approximation rate than the $P$-greedy algorithm (however they were motivated by their use for the preasymptotic range). \\
Especially for $\alpha \rightarrow 1/2$, the convergence rate can be abritrary close to $1/2$.
\item For the $f$-greedy and $f \cdot P$-greedy algorithms we have a convergence of at least $\log(n)^{1/2} \cdot n^{-1}$ respective $\log(n)^{1/2} \cdot n^{-3/4}$ according to Corollary \ref{cor:decay_rates_greedy_transl}, which is strictly better compared to the $P$-greedy algorithm.
\end{itemize}

Figure \ref{fig:beta_greedy_example} visualizes the convergence of several $\beta$-greedy algorithms for the described setting. 
One can observe that the error for the $P$-greedy algorithm ($\beta = 0$) decays approximately according to $n^{-1/2}$, which is in accordance with Eq.\ \eqref{eq:example_lower_bound_uniform}. 
For the $f$-greedy algorithm ($\beta = 1$) the error seems to decay according to $n^{-2}$, which is the fastest possible decay rate according to Eq.\ \eqref{eq:example_lower_bound_optimal}. 
For all intermediate $\beta$ values one can observe intermediate convergence rates: For values of $\beta$ closer to $1$, the error decays faster. 
The $f/P$-greedy algorithm ($\beta = \infty$) seems to give a convergence in between $n^{-1/2}$ and $n^{-2}$. \\

We remark that this behaviour of the error decay depending on $\beta$ is not unique to the Wendland $k=0$ kernel, but can also be observed for other kernels, domains and target functions $f$. 
This particular example was chosen, because it is analytically easily possible to derive several explicit statements on convergence rates for asymptotically uniform and adapted points.

\begin{figure}[h!t]
\setlength\fwidth{.7\textwidth}
\centering
\input{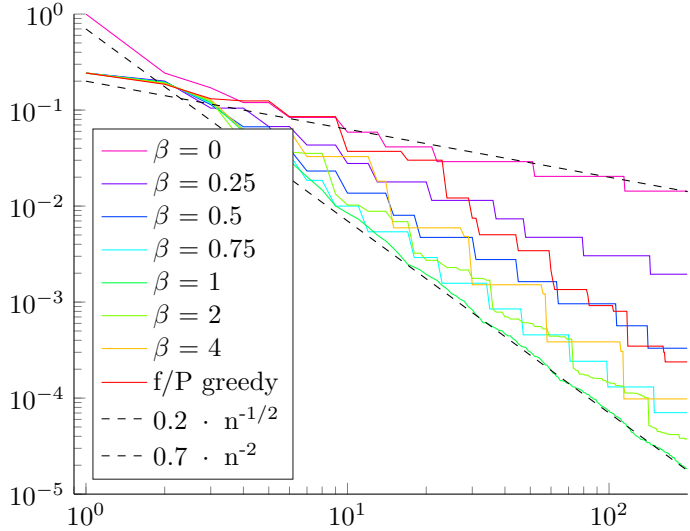}
\caption{Decay of the error $\Vert f - s_n^{(\beta)} \Vert_{L^\infty(\Omega)}$ ($y$-axis) for $\beta$-greedy algorithms in the number $n$ of chosen interpolation 
points ($x$-axis) for $\beta \in \{0, 0.25, 0.5, 0.75, 1, 2, 4, \infty\}$ and $f(x) = x^\alpha$ with $\alpha = 0.51$. Two additional dashed lines indicate a rate of convergence of $n^{-1/2}$ and 
$n^{-2}$.}
\label{fig:beta_greedy_example}
\end{figure}

\section{Conclusion and outlook} \label{sec:conclusion_outlook}

Using an abstract analysis of greedy algorithms in Hilbert spaces it was shown that arbitrary point sequences - e.g.\ generated from arbitrary greedy kernel 
algorithms - yield certain decay rates for specific power function quantities. Based on these results and refined greedy kernel interpolation analysis it was 
possible to investigate and prove convergence statements for a range of greedy kernel algorithms including the target data-dependent $f$-, $f \cdot P$- and 
$f/P$-greedy algorithms. The provided techniques and results will likely lead to further advancements, e.g.\ in the field of kernel quadrature.

Several points remain open, and they will be addressed in future research. First, the proven decay rate for the $f/P$-greedy algorithm is still not 
satisfactory and is likely improvable. Moreover, the results are independent of the special choice of function $f \in \ns$. How to make use of properties of 
that function? It would be desirable to conclude a faster decay of the $(\prod_{i=n+1}^{2n} P_i(x_{i+1}))^{1/n}$-quantity based on properties of the considered 
function $f \in \ns$. 
Finally, it is still unclear if it is possible to derive general statements on the decay of $\Vert f-s_n \Vert_{\ns}$, and what is the relationship between this 
fact and superconvergence. \\

\textbf{Acknowledgements:} The authors acknowledge the funding of the project by the Deutsche Forschungsgemeinschaft (DFG, German Research Foundation) under
Germany's Excellence Strategy - EXC 2075 - 390740016 and funding by the BMBF under contract 05M20VSA.  

\bibliography{references}
\bibliographystyle{abbrv}

\appendix

\section{Details on Subsection \ref{subsec:num_experiments_1}} \label{sec:details_example}

Consider the $k=0$ Wendland kernel and the domain $\Omega = [0, 1]$. Then the native space $\ns$ is norm equivalent to the Sobolev space $W^1_2(\Omega)$.  We 
remark that kernel interpolation using the Wendland $k=0$ kernel on centers $X_n \subset \Omega$ boils down to linear spline interpolation on the subinterval $[\min X_n, \max X_n] \subset [0, 1]$. On $\Omega \setminus [\min X_n, \max X_n]$ the interpolant is still an affine function. \\
We consider the function $f: \Omega \rightarrow \R, 
x \mapsto x^\alpha$ for some $1/2 < \alpha < 1$. For $\alpha > 1/2$ it holds $f \in W^1_2(\Omega)$, thus $f \in \ns$. 
We consider the interpolation using not-yet specified points $X_n \subset \Omega$. Define $z_1 := z_1^{(n)} := \min X_n, z_2 := z_2^{(n)} := \min X_n \setminus \{ z_1 \}$. We have for $n \geq 2$:

\begin{align*}
s_n(x) \big|_{x \in [z_1, z_2]} = \frac{z_2^\alpha - z_1^\alpha}{z_2 - z_1} \cdot (x-z_1) + z_1^\alpha
\end{align*}
\noindent We can estimate $\Vert f - s_n \Vert_{L^\infty(\Omega)}$ via:

\begin{align*}
\Vert f - s_n \Vert_{L^1([z_1, z_2])} &\leq \Vert f - s_n \Vert_{L^\infty([z_1, z_2])} \cdot | z_2 - z_1 | \\
&\leq \Vert f - s_n \Vert_{L^\infty(\Omega)} \cdot | z_2 - z_1 | \\
\Leftrightarrow \Vert f - s_n \Vert_{L^\infty(\Omega)} &\geq \Vert f - s_n \Vert_{L^1([z_1, z_2])} / | z_2 - z_1 |
\end{align*}
The integral $\Vert f - s_n \Vert_{L^1([z_1, z_2])}$ can be computed as
\begin{align*}
\Vert f - s_n \Vert_{L^1([z_1, z_2])} &= \int_{z_1}^{z_2} x^\alpha - \left( \frac{z_2^\alpha - z_1^\alpha}{z_2 - z_1} \cdot (x - z_1) + z_1^\alpha \right) ~ \mathrm{d}x \\
&= \frac{1}{1+\alpha} (z_2^{1+ \alpha} - z_1^{1+\alpha}) - \int_0^{z_2 - z_1} \frac{z_2^\alpha - z_1^\alpha}{z_2 - z_1} \cdot x + z_1^\alpha  ~ \mathrm{d}x \\
&= \frac{1}{1+\alpha} (z_2^{1+ \alpha} - z_1^{1+\alpha}) - \frac{1}{2} \cdot \frac{z_2^\alpha - z_1^\alpha}{z_2 - z_1} \cdot (z_2 - z_1)^2 - z_1^\alpha \cdot (z_2 - z_1).
\end{align*}
Thus we have
\begin{align*}
\Vert f - s_n \Vert_{L^\infty(\Omega)} &\geq \frac{\Vert f - s_n \Vert_{L^1([z_1, z_2])}}{z_2 - z_1} \\
&= \frac{1}{1+\alpha} \frac{z_2^{1+ \alpha} - z_1^{1+\alpha}}{z_2 - z_1} - \frac{1}{2} \cdot (z_2^\alpha - z_1^\alpha) - z_1^\alpha \\
&= \frac{1}{1+\alpha} \frac{z_2^{1+ \alpha} - z_1^{1+\alpha}}{z_2 - z_1} - \frac{1}{2} \cdot (z_2^\alpha + z_1^\alpha) \\
&= z_2^\alpha \cdot \left(  \frac{1}{1+\alpha} \frac{z_2 - z_1 z_1^{\alpha} z_2^{-\alpha}}{z_2 - z_1} - \frac{1 + z_1^\alpha z_2^{-\alpha}}{2} \right) \\
&= z_2^\alpha \cdot \left(  \frac{1}{1+\alpha} \frac{z_2 - z_1(1 + z_1^{\alpha} z_2^{-\alpha} - 1)}{z_2 - z_1} - \frac{1 + z_1^\alpha z_2^{-\alpha}}{2} \right) \\
&= z_2^\alpha \cdot \left( \frac{1}{1+\alpha} + \frac{1}{1+\alpha} \frac{z_1(1 - z_1^{\alpha} z_2^{-\alpha})}{z_2 - z_1} - \frac{1 + z_1^\alpha z_2^{-\alpha}}{2} \right) \\
&= z_2^\alpha \cdot \left( \frac{1}{1+\alpha} + \frac{1}{1+\alpha} \frac{1 - z_1^{\alpha} z_2^{-\alpha}}{z_2z_1^{-1} - 1} - \frac{1 + z_1^\alpha z_2^{-\alpha}}{2} \right).
\end{align*}
We proceed by setting $k := z_2/z_1$:
\begin{align*}
\Vert f - s_n \Vert_{L^\infty(\Omega)} &\geq z_2^\alpha \cdot \underbrace{\left( \frac{1}{1+\alpha} + \frac{1}{1+\alpha} \frac{1 - k^{-\alpha}}{k - 1} - \frac{1 + k^{-\alpha}}{2} \right)}_{=: h_\alpha(k)}.
\end{align*}
We consider asymptotically uniform points, i.e.\ $\exists C > 0 ~ \forall n \in \N ~ h_n / q_n \leq C$. Based on the definition of the separation and fill distance we can estimate
\begin{align*}
k \equiv \frac{z_2}{z_1} \geq \frac{z_1 + q_n}{z_1} \geq 1 + \frac{q_n}{z_1} \geq 1 + \frac{q_n}{h_n}
\end{align*}
Using the asymptotic uniformity, we have finally
\begin{align*}
k \geq 1 + \frac{q_n}{h_n} \geq 1 + C^{-1} > 1.
\end{align*}
Finally an analysis of the 1D function $h_\alpha: [1 + C^{-1}, \infty) \rightarrow \R$ shows that it holds 
\begin{align*}
h_\alpha(k) \geq h_\alpha(1 + C^{-1}) > 0
\end{align*}
for $k \in [1 + C^{-1}, \infty)$. This finally implies
\begin{align*}
\Vert f - s_n \Vert_{L^\infty(\Omega)} &\geq z_2^\alpha \cdot \min_{k \in [1+C^{-1}, \infty)} h_\alpha(k) \\
&\geq c \cdot n^{-\alpha}
\end{align*}
for some $c > 0$ due to $z_2 \geq q_n \geq \tilde{c} n^{-1}$.

\end{document}